\documentclass[11pt,letterpaper]{amsart}

\usepackage{custom-template, shortcuts,graphicx, setspace, float}
\usepackage[thinlines]{easytable}
\usepackage{makecell}
\makeatletter
\@namedef{subjclassname@2020}{\textup{2020} Mathematics Subject Classification}
\makeatother

\title{Experiments on $3$-isogeny Selmer groups of elliptic curves with a $3$-torsion point}
\author{Ariel Weiss}
\address{Ariel Weiss, Department of Mathematics, Trinity College, Hartford CT\vspace{-5pt}}
\email{ariel.weiss@trincoll.edu}
\author {Dongchen Zou}
\address{Dongchen Zou, University of Chicago, Chicago IL\vspace{-5pt}}
\email{dcz0711@uchicago.edu}
\date{}

\begin{document}
	\maketitle
	\begin{abstract}
		Let $\eab\:y^2 +Axy + By = x^3$ be an elliptic curve over $\Q$. The $3$-torsion point $(0,0)$ induces a $3$-isogeny $\phi\:\eab\to\eabh$. Assuming $\eab$ has good reduction at $3$, we construct an explicit $(m+t)\times m$ matrix $\mab'$ over $\F_3$, whose kernel encodes the dual isogeny Selmer group $\Sel_{\hat\phi}(\eabh)$ modulo the image of the torsion point $(0,0)$. Here, $m = \omega(B) - 1$, and $t$ encodes the \emph{global Selmer ratio} $3^{t-2}$. 
		
		We compute $\mab'$ in various regimes for billions of elliptic curves $\eab$. Based on our data, and motivated by prevalence of random linear algebraic models throughout number theory, we conjecture that, for fixed $m$ and $t$, the matrices $\mab'$ become uniformly distributed, and we formulate a corresponding conjecture for the distribution of $\Sel_{\phi}(\eab)$. Our model predicts that, for fixed $t$, the average size of $\Sel_{\phi}(\eab)$ is $1 + 3^t$.
	\end{abstract}
	
	\section{Introduction}
	
	Let $\E$ be a family of elliptic curves. A fundamental problem in arithmetic statistics is to understand the \emph{average} behaviour of curves  in $\E$. How are their ranks and Selmer ranks distributed? What can be said about their Tate--Shafarevich groups and other arithmetic invariants? For the family of all elliptic curves over $\Q$, the conjectural answers to these questions are expressed through random linear algebra \cite{delaunay2,Delaunay,PoonenRains,BKLPR,ppvw}. Random linear algebra plays a similar role elsewhere in number theory. For example, the Cohen--Lenstra heuristics \cite{cohenlenstra} model class groups by
	random finite abelian groups, and the Friedman--Washington model \cite{friedmanwashington} realises
	closely related distributions as cokernels of random $p$-adic matrices. The recurring principle is that once the visible arithmetic constraints have been imposed, the remaining statistics should be governed by a natural random linear-algebraic object.
	
	This paper asks what that principle should mean when every curve in the family $\E$ carries an isogeny. If $\phi\:E\to E'$ is an isogeny and $\widehat\phi$ its dual, then the Selmer groups $\Selp$ and $\dSelp$ are linked by global duality. Indeed, by a formula of Cassels \cite{CasselsVIII}, we have 
	\begin{equation}\label{eq:greenberg-wiles}
		\frac{\#\Selp}{\#\dSelp} = \frac{\#E[\phi](\Q)}{\#E'[\widehat\phi](\Q)}c(\phi, E)
	\end{equation}
	where $c(\phi, E)$ denotes the \emph{(global) Selmer ratio}, a product $c(\phi, E) = \prod_{v\le \infty}c_v(\phi, E)$, where
	\[c_v(\phi, E) = \frac{\#\coker(E(\Q_v)\xrightarrow{\phi}E'(\Q_v))}{\#\ker(E(\Q_v)\xrightarrow{\phi}E'(\Q_v))}\]
	is the \emph{local Selmer ratio}. Thus, the two groups $\Selp$ and $\dSelp$ cannot be modelled independently. The distribution of the global Selmer ratio can have a decisive impact on the distributions of $\Selp$ and $\dSelp$. In many natural families, its variation can force one of the two isogeny Selmer groups to become typically large and make its average size and higher moments infinite \cite{klags-lo1, klags-lo2,chan-h-li,abs,Chan_2023,Chan_2025,phillips}. Koymans--Smith \cite{koymans-smith} conjecture more generally that \textit{all} unbounded Selmer moments arise from constraints imposed by global duality.
	These results suggest that the global Selmer ratio should be treated as one of the visible arithmetic constraints in the principle above: in order for the distribution of $\Selp$ and $\dSelp$ to be governed by linear algebra, we should first control for the global Selmer ratio.
	
	We investigate this principle experimentally in the family $\E$ of elliptic curves over $\Q$ with a rational point of order $3$. Every curve in $\E$ has a unique normalised model
	\[\eab\:y^2 + Axy + By = x^3\]
	with $A, B\iZ$, $B>0$ and $p^3\nmid B$ whenever $p\mid A$ (see \Cref{sec:elliptic curve}). The $3$-torsion point $(0,0)$ induces a $3$-isogeny $\phi\:\eab\to\eabh$. We let $H(\eab) = \max\set{|A|^3, B}$ be the height of $\eab$.
	
	This family is a particularly revealing test case: Bhargava--Ho prove that its average rank is bounded \cite{BhargavaHoII}, while recent work of Chan--Verzobio, Phillips, and Koymans--Smith shows that variation in the global Selmer ratio leads to arbitrarily large Selmer groups and infinite Selmer moments \cite{Chan_2025,phillips,koymans-smith}. We conjecture that, after controlling for the visible arithmetic constraints---torsion, and the global Selmer ratio---the remaining Selmer distribution is governed by random linear algebra. The particular random-matrix model comes from an explicit $3$-descent: after removing the class forced by rational torsion, we show that the $3$-isogeny Selmer groups are described by the kernels of an explicit matrix and its transpose.
	
	\begin{conjecture}\label{conj:intro}
		Fix an integer $t$, and let $\E_t$ denote the set of elliptic curves $E\in\E$ with global Selmer ratio $c(\phi,E)=3^{t-2}$.  Let $\E_t(X)$ denote the set of elliptic curves $E\in\E_t$ with height at most $X$. Then for every integer $d\ge \max(0, -t)$
		\[\lim_{X\to\infty}\frac{\#\set{E\in \E_t(X) :\dim\dSelp = d+1}}{\#\E_t(X)} = \P_t(d),\]
		where $\P_t(d)$ is the limit, as $m\to \infty$ of the probability that a random $(m+t)\times m$ matrix over $\F_3$ has nullity $d$.
	\end{conjecture}
	
	The indexing in the conjecture incorporates the two constraints we are controlling for. Indeed, we have $\#E[\phi](\Q) = 3$ and $\#E'[\widehat\phi](\Q)=1$, so \eqref{eq:greenberg-wiles} gives
	\begin{equation}\label{eq:tam}
		\dim_{\F_3}\Selp - \dim_{\F_3}\dSelp = t - 1
	\end{equation}
	for every $E\in \E_t$. When $B$ is not a cube, the $3$-torsion point $(0,0)\in\eab(\Q)$ maps to a non-zero class in $\Sel_{\hat\phi}(\eabh)$. Therefore, accounting for this class, if $d = \dim\Sel_{\hat\phi}(\eabh) -1$, then $d\ge \max(0,-t)$, which is also the minimum nullity of an $(m+t)\times m$ matrix. 
	
	\Cref{conj:intro} asserts that the distribution of $\dSelp$ for $E\in \E_t$ matches a random matrix distribution. In fact, our experiments support the more refined conjecture that the distribution of $\dSelp$ is actually determined by random matrices.
	We begin with the explicit 3-descent of Cohen--Pazuki \cite{Cohen_2009} and the local descriptions used by Chan \cite{Chan_2023}. Our construction relies on the identification $H^1(\Q, \eabh[\widehat\phi]) \simeq H^1(\Q, \mu_3)\simeq \qmc$, which allows us to identify every $\Sel_{\hat\phi}(\eabh)$ as a subgroup of $\qmc$. Under this identification, the Kummer map maps the torsion point $(0,0)$ to $[B]\ii\in\qmc$.
	Suppose that $\eab\in\E_t$ has good reduction at $3$ and that neither $B$ nor $A^3-27B$ is a rational cube. In \Cref{sec:matrix}, we construct an explicit matrix $\mab$ over $\F_3$ with dimensions $(m+t)\times (m+1)$, where $m =\omega(B) -1$, and prove in \Cref{thm:ker',thm: matrix cols} that
	\[\nullity\mab = \dim_{\F_3}\br{\Sel_{\hat\phi}(\eabh)},\]
	and that
	\[\nullity\mab^T = \dim_{\F_3}\Sel_\phi(\eab).\]
	The columns of $\mab$ are indexed by the primes dividing $B$, and its rows are indexed by the primes $p\mid (A^3 -27B)$ such that $p\equiv 1\pmod 3$. The torsion class $(0,0)$ gives a known vector in the kernel and deleting a column on which this vector is non-zero gives a $(m+t)\times m$ matrix $\mab'$ whose nullspace encodes $\Sel_{\hat\phi}(\eabh)/\langle [B]\rangle$. The entries of $\mab'$ are determined by cubic residue symbols in terms of the primes that index the rows and columns.
	
	For fixed $m$ and $t$, we conjecture that the matrices $\mab'$ are equidistributed in $\M_{(m+t)\times m}(\F_3)$ as the height of $\eab$ goes to infinity (see \Cref{conj:componentwise}). This conjecture is more refined than \Cref{conj:intro}, which records only the resulting distribution of nullities. Since $m=\omega(B)-1$ tends to infinity with the height of a typical curve, a sufficiently uniform form of this fixed $m$ conjecture would imply \Cref{conj:intro} on the subfamily where $\mab'$ is defined. We explain in \Cref{sec:conjectures} why we expect the same limiting distribution for the whole of $\E_t$.
	
	Our experiments test this conjecture directly: for fixed $m$ and $t$, we compute $\mab'$ for large collections of curves, and compare the frequency of each matrix with the uniform distribution. Thus, our experiments test the underlying principle that, once the visible arithmetic constraints have been accounted for, the remaining variation is governed by random linear algebra.
	
	\Cref{conj:intro} has several natural consequences. First, via \eqref{eq:tam}, \Cref{conj:intro} also determines the distribution of $\Selp$: for every integer $d \ge \max(0, t)$
	\[\lim_{X\to\infty}\frac{\#\set{E\in \E_t(X) :\dim\Selp = d}}{\#\E_t(X)} = \P_t(d-t).\]
	
	The probabilities $\P_t(d)$ are well-studied and admit explicit formulae \cite{fulman-goldstein}*{Sec.~3}. The corresponding first-moment prediction is particularly simple:
	
	\begin{corollary}\label{thm:avg-intro}
		Assume that the convergence in \Cref{conj:intro} is sufficiently uniform to pass to first moments. Then 
		\[\lim_{X\to \infty}\br{\avg_{E\in\E_t(X)}\#\Selp} = 1 + 3^t.\]
	\end{corollary}
	
	Analogues of this first moment prediction have been proven for several twist families of elliptic curves and abelian varieties \cite{bkls,elk,ShnidmanWeiss} using Bhargava's geometry-of-numbers techniques. Other closely related results in the isogeny-Selmer setting include the work of Kane--Thorne \cite{KaneThorne} and Smith \cite{smith2025}. Our viewpoint is closer to that of Kane--Klagsbrun \cite{kane-klagsbrun}, and builds on the residue-symbol matrix methods of Heath-Brown and Swinnerton-Dyer \cite{heath-brown-1,Heath-Brown,swinnerton-dyer-heuristics}.

	\subsection{Our experiments}
	
	Our computations test the prediction that the matrices $\mab'$ are equidistributed. We compute the reduced matrices $M'_{A, B}$ for large collections of curves $\eab\:y^2 + Axy + By = x^3$ with good reduction at $3$, and then group the curves according to the dimensions of $M'_{A,B}$. For each fixed pair of dimensions, we compare the observed frequencies of the matrices $M'_{A,B}$ with the uniform distribution on matrices of that size.
	
	The most natural first approach is to compute $M'_{A, B}$ for all elliptic curves $\eab\in \E$ of bounded height, i.e.\ with $H(\eab) = \max(|A|^3, |B|)$ less than some fixed parameter $X$. This approach yielded matrices that were \emph{not} uniformly distributed. In hindsight, this was to be expected. The dimensions of $M'_{A, B}$ are controlled by the number of distinct prime factors of $B$ and $A^3-27B$. For example, the number of columns is exactly $\omega(B)-1$. By the Erd\H{o}s--Kac theorem, a generic integer of size roughly $X$ has about $\log\log X$ prime factors. Therefore, to naturally see matrices of moderately large dimension, one would have to take $X$ to be astronomically large. For example, for a generic $M'_{A, B}$ to have just $4$ columns, one would already have to take $X$ on the order of $10^{64}$. At computationally accessible heights, matrices of moderate size therefore arise disproportionately from values of $B$ with many small prime factors. Since these small primes occur in the cubic residue symbols defining $\mab'$, there is no reason for the resulting matrices to be uniformly distributed. Such values of $B$ should not affect the limiting distribution: the set of integers $B<X$ that are supported on primes $p< X^{o(1)}$ has density $0$. Similarly slow convergence also appears in the work of Kane--Klagsbrun on $2$-isogeny Selmer groups in quadratic twist families: they prove a distribution analogous to that of \Cref{conj:intro}, but with error terms that decay only as powers of $\log\log X$ \cite{kane-klagsbrun}. Hence, even when the limiting random-matrix model is correct, convergence may be extremely slow.
	
	We therefore use two different sampling methods in order to model the effect of choosing $X$ to be very large. 
	
	\subsubsection{Sifted height experiments}
	
	Our first experiment stays as close as possible to the usual height ordering, while removing the small primes that dominate the matrices at computationally accessible heights. We exhaust all elliptic curves $\eab$ whose heights lie in a short interval
	\[h_0 \le H(\eab)\le h_1,\]
	subject to the condition that $B(A^3-27B)$ has no prime factor below a chosen cutoff. Restricting to a short interval allows us to work at substantially larger heights without first enumerating all smaller curves, while the sieve suppresses the small-prime effects that obstruct equidistribution. This method has the advantage that it remains close to the height ordering, and gives an exhaustive search rather than a random sample.
	
	For example, in experiment H-3250 (see \Cref{tab:H3250-statistics}), we compute $\mab'$ for every $\eab$ with $H(\eab)\in[3250^3,3260^3]$ and $B(A^3-27B)$ not divisible by any prime $<1000$. The three largest non-trivial blocks contain approximately $4.4$ billion $1\times1$ matrices, $1.3$ billion $2\times1$ matrices, and $204$ million $1\times2$ matrices. Every individual matrix in these blocks occurs within $0.21\%$ of its uniform prediction, and the maximum difference between the observed and uniform probability of any given nullity is $6\cdot 10^{-4}$. On the other hand, the observed standard deviations are much larger, which shows that these counts do not yet behave like independent uniform draws. The observed distribution is nevertheless consistent with uniformity at a slow rate of convergence.
	
	These experiments also give evidence for the role of small primes in affecting the distribution. In four different experiments, we compute $\mab'$ for all $\eab$ with $H(\eab)\in [1000^3,1010^3]$, but with four different prime cutoffs (see \Cref{tab:height-sieve-comparison}). The experiments show that as the smallest prime divisor $p_{\min}$ of $B(A^3-27B)$ increases through the ranges $[100,300)$, $[300,500)$, and $[500,800)$, there is a clear, rapid improvement in the uniformity of $\mab'$. This improvement doesn't persist at the same rate as $p_{\min}$ increases beyond these ranges, which suggests that it is small primes that have the largest influence. As $H(\eab)\to\infty$, the influence of these small primes will become more and more diluted.
	
	The main limitation of the sifted height method is that the sieve makes large matrices rare. To probe matrix spaces of substantially higher dimension, we therefore use a second experiment in which the number of prime factors of $B$ is fixed in advance.
	
	\subsubsection{Prime factor experiments}
	
	Fix integers $N$ and $n$. We choose $n$ distinct primes $p_1,\ldots,p_n$ from the first $N$ primes greater than $3$, choose their exponents $e_i$ according to the distribution of the positive $p_i$-adic valuation of a random integer, and set
	\[B=\prod_{i=1}^n p_i^{e_i}.\]
	We then choose $A$ randomly on the natural scale $|A|\asymp B^{1/3}$ such that $3\nmid A^3-27B$. In this way, the number of columns of $\mab'$ is fixed to be $n-1$, while the primes governing its cubic residue symbols are allowed to vary through a large range. We discard matrices with more than $12$ entries, since beyond this point there are too many possible matrices for our sample sizes to meaningfully test their individual frequencies.
	
	This sampling procedure is modelled on the ordering introduced by Swinnerton-Dyer in his study of $2$-Selmer groups in quadratic twist families \cite{swinnerton-dyer-heuristics}. Rather than ordering twists by size, Swinnerton-Dyer fixes the number of prime factors of the twisting parameter and lets those primes vary. Kane later showed how to pass from such a fixed-prime-factor ordering to the natural ordering by height, and Kane--Klagsbrun use a closely related strategy for isogeny Selmer groups \cite{Kane,kane-klagsbrun}. Our factor experiment is designed to test the corresponding fixed-factor prediction in the present cubic setting.
	
	The resulting agreement with the uniform matrix model is striking. In experiment F10k-7, for example, we take $n=7$ and choose the prime factors of $B$ from a pool of $10{,}000$ primes (see \Cref{tab:F10k7-statistics}). The experiment produced $52{,}656{,}296$ $2\times 6$ matrices, distributed among all $3^{12} = 531{,}441$ possible matrices. The data is extremely close to what one would expect from independent uniform draws: every matrix occurs; the observed standard deviation is $9.97$ compared to the uniform standard deviation of $9.95$; the maximum difference between the observed and uniform probability of any given nullity is $2.283\cdot 10^{-5}$; and the discrepancies between the observed counts and the mean, while large, are consistent with uniform draws over $3^{12}$ matrices. The same behaviour persists across the other factor experiments, and uniformity improves as the size of the prime pool increases.
	
	We give a detailed description of both experiments and an analysis of the results in \Cref{sec:results}. We give additional tables in \Cref{appendix}. Our SageMath code and raw data are available at our \href{https://github.com/dcz0711/3-Selmer-Groups/}{Github repository}.

	\section{The family of elliptic curves with a $3$-torsion point}
	
	\subsection{Elliptic curves with a $3$-torsion point}\label{sec:elliptic curve}
	
	Let $\E$ denote the set of elliptic curves over $\Q$ with a rational $3$-torsion point. Then, after a change of variables, every $E\in\E$ is isomorphic to a curve
	\[\eab\: y^2 + Axy + By = x^3\]
	for some integers $A, B$, such that the discriminant $\Delta_{A, B} = B^3(A^3 - 27B)$ is non-zero.
	Since $E_{pA, p^3B}$ and $\eab$ are isomorphic, we normalise our curves so that $B>0$ and whenever a prime $p$ divides $A$, $p^3\nmid B$. 
	
	\begin{definition}\label{def:height}
		We define the \emph{height} of $\eab$ to be $\max(|A|^3, |B|)$.
	\end{definition}
	
	The elliptic curve $\eab$ has a $3$-torsion point $(0,0)$. Quotienting out by this point gives a $3$-isogeny
	\[\phi\: \eab\to \eabh,\]
	where
	\[\eabh \: Y^2 = X^3 - 3\br{\frac A2X - \frac1{18}(A^3 - 27B)}^2.\]

	The rational map $\phi$ is given, away from its kernel, by
	\[
	\phi(x,y)=\left(
	\frac{x^3+\frac{A^2}{3}x^2+ABx+B^2}{x^2},
	\frac{(2y+Ax+B)(x^3-ABx-2B^2)}{2x^3}
	\right),
	\]
	and $\phi$ sends $\mathcal O$, $(0,0)$, and $(0,-B)$ to $\widehat{\mathcal O}$.
	For the dual map, set
	\[
	C=\frac{A^3-27B}{9}.
	\]
	If $(X,Y)\in \eabh$ and $X\neq 0$, define
	\[
	x_0=\frac{X^3-A^2X^2+3ACX-3C^2}{9X^2},
	\qquad
	Y_0=\frac{Y(X^3-3ACX+6C^2)}{27X^3}.
	\]
	Then
	\[
	\widehat\phi(X,Y)=\left(x_0,\,Y_0-\frac{Ax_0+B}{2}\right),
	\]
	with the points in $\ker(\widehat\phi)$ sent to $\mathcal O$.
	
	The elliptic curve $\eab$ has discriminant $B^3(A^3 - 27B)$. 
	
	\begin{lemma}
		\label{lemma:good_reduction_at_3}
		The elliptic curve $\eab$ has good reduction at the prime $p=3$ if and only if $3 \nmid AB$.
	\end{lemma}
	\begin{proof}
		An elliptic curve given by a minimal Weierstrass equation has good reduction at a prime $p$ if and only if its discriminant is not divisible by $p$. Our assumption that $27\nmid B$ if $3\mid A$ ensures that the equation is minimal at $3$. Note that 
		\(B^3(A^3-27B) \equiv (AB)^3 \pmod{3}\) and by Fermat's Little Theorem, $(AB)^3 \equiv AB \pmod{3}$. Thus $\eab$ has good reduction at $3$ if and only if $3\nmid AB$.
	\end{proof}
	
	Depending on $A$ and $B$, the elliptic curve $\eab$ may admit a second $3$-isogeny, or $\phi$ could extend to a $9$-isogeny. Analogously to the cases considered in \cite{kane-klagsbrun,smith2025} for $2$-isogenies, we expect that the statistics of these curves will be different to those of a generic curve $E\in \E$. The following lemma shows that elliptic curves $E\in\E$ with additional isogenies are both rare and easy to detect.
	
	\begin{lemma}\label{lem:extra-isogenies}
		The following are equivalent:
		\begin{enumerate}[leftmargin=*]
			\item The $3$-power isogeny class of $\eab$ consists of exactly $\eab$ and $\eabh$.
			\item Neither $B$ nor $A^3 -27B$ is a rational cube.
		\end{enumerate}
	\end{lemma}
	
	\begin{proof}
		The $3$-power isogeny class of $\eab$ consists of additional curves only if $\eab$ admits a second $3$-isogeny, or if $\phi$ extends to a $9$-isogeny.
		
		In the former case, the quadratic twist $(\eab)_{-3}$ must have a $3$-torsion point. The elliptic curve $(\eab)_{-3}$ is isomorphic to $y^2 = x^3 -3\br{\frac A2 x -\frac {3B}2}^2$. By \cite{Cohen_2009}*{Lem.~2.1}, it follows that $(\eab)_{-3}$ has a $3$-torsion point if and only if $A^3-27B$ is a cube.
		
		In the latter case, $\eabh$ must have a $3$-torsion point. But by \cite{Cohen_2009}*{Lem.~2.1}, $\eabh\:Y^2 = X^3-3\br{\frac A2 X - \frac1{18}(A^3 - 27 B)}^2$ has a rational $3$-torsion point if and only if $27B$ is a cube.
	\end{proof}
	
	The exceptional curves are thin in the full family $\E$. Indeed, $\#\E(X)\asymp X^{4/3}$, while there are $O(X^{2/3})$ pairs $(A, B)$ with $\max{|A|^3, B}< X$ and either $B$ or $A^3 -27B$ a cube. Thus, we expect that these families will not affect the distribution in \Cref{conj:intro}. We remove these families from our matrix experiments, since their additional isogenies introduce further structure.
	
	\subsection{The isogeny Selmer group}
	
	Let $\phi\: E\to E'$ be an isogeny of elliptic curves over $\Q$, and let $G_{\Q} = \Ga\Q$. For any place $v$ of $\Q$, let $\Q_v$ be the corresponding local field. The short exact sequence of $G_{\Q}$-modules 
	\[0 \to E[\phi] \to E \xrightarrow{\phi} E' \to 0\]
	induces a local Kummer map for every place $v$:
	\[
	\delta_v: E'(\Q_v) / \phi(E(\Q_v)) \hookrightarrow H^1(\Q_v, E[\phi]).
	\]
	Similarly, for the dual isogeny $\phih$, the short exact sequence
	\[0\to E'[\phih]\to E'\xrightarrow{\phih}E\to 0\]
	induces a local Kummer map
	\[
	\widehat{\delta}_v: E(\Q_v) / \phih(E'(\Q_v)) \hookrightarrow H^1(\Q_v, E'[\phih]).
	\]
	
	\begin{definition}
		The $\phi$-Selmer group, $\sel_{\phi}(E)$, is defined as the subgroup of classes in the global cohomology group $H^1(\Q, E[\phi])$ that locally restrict to the image of the Kummer map at every place $v$. Formally, 
		\[\Selp = \left\{ \xi \in H^1(\Q, E[\phi]) \;\middle|\; \text{res}_v(\xi) \in \text{Im}(\delta_v) \text{ for all } v \right\}. \]
	\end{definition}
	Similarly, the $\phih$-Selmer group, $\Sel_{\phih}(E')$, is defined as the subgroup of classes in $H^1(\Q, E'[\phih])$ restricting to the local Kummer images:
	\[\dSelpd = \left\{ \xi \in H^1(\Q, E'[\phih]) \;\middle|\; \text{res}_v(\xi) \in \text{Im}(\hat{\delta}_v) \text{ for all } v \right\}.\]
	
	We now consider the case that $E = \eab$ and $\phi$ is the $3$-isogeny $\phi\:\eab\to\eabh$ from the previous section.
	
	\begin{lemma}
		\label{where_selmer_groups_live}
		The $\phih$-Selmer group $\sel_{\phih}(\eabh)$ embeds into $\msc\Q$.
	\end{lemma}
	\begin{proof}
		By definition, the kernel of the isogeny $\phi\: \eab \to \eabh$ is $E[\phi] = \{\mathcal{O}, P, -P\}$, where $P = (0,0)$. Because the points in $E[\phi]$ are defined over $\Q$, the Galois group $G_{\Q}$ acts trivially on $E[\phi]$. Hence, as a $G_{\Q}$-module, $E[\phi] \cong \Z/3\Z$. The dual isogeny $\phih\: E' \to E$ has kernel $E'[\phih]$. The Weil pairing provides a perfect $G_\Q$-equivariant pairing
		\[e_{\phi} \: E[\phi] \times E'[\phih] \to \mu_3,\]
		where $\mu_3$ is the group of third roots of unity. Because $G_{\Q}$ acts trivially on the first argument $E[\phi]$, we must have $E'[\phih] \cong \mu_3$. 
		By definition, the $\phih$-Selmer group $\Sel_{\phih}(E'/\Q)$ is the subgroup of the global cohomology group $H^1(\Q, E'[\phih])$ consisting of classes that satisfy local solubility conditions at all places. Since $E'[\phih] \cong \mu_3$, the Selmer group sits inside $H^1(\Q, \mu_3)$. To identify $H^1(\Q, \mu_3)$ explicitly, consider the short exact sequence of $G_{\Q}$-modules:
		\[1 \to \mu_3 \to\Qb\t\xrightarrow{x\mapsto x^3}\Qb\t\to 1\]
		Taking Galois cohomology yields the associated long exact sequence:
		$$ \Q^\times \xrightarrow{x \mapsto x^3} \Q^\times \xrightarrow{\delta} H^1(\Q, \mu_3) \longrightarrow H^1(\Q, \overline{\Q}^\times). $$
		By Hilbert's Theorem 90, the cohomology group $H^1(\Q, \overline{\Q}^\times)$ is trivial. Therefore, the connecting homomorphism $\delta$ is surjective, and its kernel is precisely the image of the cubing map, yielding an isomorphism
		\[H^1(\Q, \mu_3) \cong \Q^\times / (\Q^\times)^3.\]
		This completes the proof.
	\end{proof}
	
	Thus, for every $A,B$, we may view the $\phih$-Selmer group as a subgroup of $\qmc$. The same argument identifies 
	\[H^1(\Q_v, \eabh[\phih])\cong \msc{\Q_v}.\]
	With this identification, we can define the local Kummer map $\widehat\delta_v\:\eab(\Q_v)\to \msc{\Q_v}$ explicitly: for any point $Q = (x, y)\in \eab(\Q_v)$, we have
	\[\widehat\delta_v(Q) = \begin{cases}
		1 \pmod{\Q_v^{\times 3}} & \text{if } Q = \mathcal{O} \\
		B\ii \pmod{\Q_v^{\times 3}} & \text{if } Q = (0,0)\\
		y \pmod{\Q_v^{\times 3}} & \text{if } \text{otherwise}
	\end{cases}\]
	Similarly, the global Kummer map
	\[\widehat\delta\:\eab(\Q)\to \qmc\]
	is given by 
	\[
	\widehat\delta(Q) = \begin{cases} 
		1 \pmod{\Q^{\times 3}} & \text{if } Q = \mathcal{O} \\
		B\ii \pmod{\Q^{\times 3}} & \text{if } Q = (0,0)\\
		y \pmod{\Q^{\times 3}} & \text{if } \text{otherwise}
	\end{cases} 
	\]
	
	\begin{lemma}
		\label{guaranteed_point_in_selmer}
		The class of $B$ is guaranteed to be in $\Sel_{\hat\phi}(\eabh)$.
	\end{lemma}
	\begin{proof}
		Let $Q = (0,0)$. Then $B = \widehat\delta(-Q)$.
	\end{proof}
	
	In particular, under our running assumption that $B$ is not a cube, we always have $\dim\Sel_{\phih}\eabh \ge 1$.
	
	\subsection{The global Selmer ratio}
	
	Let $\phi\:E\to E'$ be an isogeny of elliptic curves over $\Q$.
	
	\begin{definition}\label{selmer-ratio}
		For each place $v$, the \emph{local Selmer ratio} is
		\[c_v(\phi, E) = \frac{\#\coker(E(\Q_v)\xrightarrow{\phi}E'(\Q_v))}{\#\ker(E(\Q_v)\xrightarrow{\phi}E'(\Q_v))}.\]
		The global Selmer ratio is
		\[c(\phi, E) = \prod_{v}c_v(\phi, E).\]
	\end{definition}

	Note that if $E$ has good reduction at $v$ and if $v$ does not divide $\infty$ or the degree of $\phi$, then $c_v(\phi, E) = 1$. Thus, the product defining $c(\phi, E)$ is finite. If $\phi$ is a $p$-isogeny, then $c(\phi, E)$ is a power of $p$.
	
	The relation between the Selmer ratio and Selmer groups goes back to Cassels \cite{CasselsVIII}. We use it in the form of the Greenberg Wiles  theorem (see \cite{NeukirchCohomology}*{8.7.9}):
	
	\begin{theorem}[Greenberg--Wiles]\label{thm:greenberg-wiles}
		We have
		\[c(\phi, E) = \frac{\#\Selp}{\#\dSelp}\cdot\frac{\#E'[\widehat\phi](\Q)}{\#E[\phi](\Q)}.\]
	\end{theorem}
	
	\begin{remark}
		There are several different conventions for the Selmer ratio. In \cite{kane-klagsbrun}, the authors refer instead to the \emph{Tamagawa ratio}, which they define as 
		\[\frac{\#\Selp}{\#\dSelp}.\]
		In their specific setting, $\#E[\phi](\Q) = \#E'[\widehat\phi](\Q) = 2$, so their definition agrees with ours. Several other authors keep this definition for the general case (\cite{phillips,koymans-smith}). Our convention follows that of \cite{bkls}, which ensures that the global Selmer ratio is a product of local factors.
	\end{remark}
	
	For our isogeny, the two rational kernel sizes are
	\[\#\eab[\phi](\Q) = 3,\qquad\#\eabh[\widehat\phi](\Q)=1,\]
	so the torsion factor in \Cref{thm:greenberg-wiles} is $\frac13$. Indeed, as a Galois module, $\eabh[\widehat\phi]$ is isomorphic to $\mu_3$.
	
	When $\eab$ has good reduction at $3$, the local Selmer ratios can be read off from $B$ and $A^3-27B$.
	
	\begin{proposition}
		Assume that $3\nmid AB$. Then the local Selmer ratio $c_p(\phi, \eab)$ is given by the following table:
		\begin{center}
			\begin{TAB}(c,1cm,1cm)[5pt]{|c|c|}{|c|c|c|c|c|c|c|}
				Case & $c_p(\phi, \eab)$    \\
				$p\ne 3, p\nmid B(A^3 - 27B)$ & $1$  \\
				$p\ne 3, p\nmid B, p\mid A^3-27B$& \makecell{ $3 \quad\quad \text{ if } p \equiv 1 \pmod 3$ \\ $1\quad\quad \text{ if } p \equiv 2 \pmod 3$}\\
				$p\neq 3$, $p \mid B$, $p\nmid A^3 - 27B$& $1 / 3$\\
				$p\neq 3$, $p \mid B$, $p\mid A^3-27B$& \makecell{ $1 \quad\quad \text{ if } p \equiv 1 \pmod 3$ \\ $1/3\quad \text{ if } p \equiv 2 \pmod 3$}\\
				$p=3$& $1$\\
				$\infty$ & $1 / 3$
			\end{TAB}
		\end{center}
	\end{proposition}
	
	\begin{proof}
		At a prime of good reduction away from $3$, the local Selmer ratio is always $1$. Cases $2$ and $3$ follow from \cite{Shnidman_2023}*{Lem.~5.10} with $d = 1$. The case $p = \infty$ is \cite{Shnidman_2023}*{Lem.~5.9} and the case $p = 3$ is \cite{Shnidman_2023}*{Lem.~5.8}.
		
		It remains to consider the case that $p\mid B$ and $p\mid A^3 -27B$. By our normalisation of $(A, B)$, we must have $v_p(B) = 1$ or $2$. Tate's algorithm \cite{Silverman2014Advanced}*{Chap.~IV} shows that $\eab$ has Kodaira type $IV$ if $v_p(B) = 1$ and type $IV^*$ if $v_p(B) = 2$. 
		
		Since $p\ne 3$, by \cite{schaefer}*{Lem.~3.8}, the local Selmer ratio is the quotient of the local Tamagawa numbers of $\eabh$ and $\eab$. The $IV$/$IV^*$ row of \cite{Dokchitser-local}*{Table 1} therefore shows that $c_p(\phi, \eab)$ is $1$ if $\mu_3\sub \Q_p$ and $\frac 13$ otherwise, which gives the fourth row.
	\end{proof}

	For a non-zero integer $n$, let $\omega(n)$ be the number of its distinct prime factors, and let $\omega_1(n)$ be the number of its distinct prime factors that are $1\pmod 3$. We deduce the following corollary.
	
	\begin{corollary}\label{prop:selmer ratio}
		Assume that $3\nmid AB$. Write $c(\phi, \eab) = 3^{t_{A,B}}$. Then 
		\[t_{A,B} = -1 - \omega(B) + \omega_1(A^3 - 27B).\]
	\end{corollary}

	\begin{remark}
		It follows from \Cref{prop:selmer ratio} and \Cref{thm:greenberg-wiles} that
		\[\dim_{\F_3}\Sel_\phi(\eab) -\dim_{\F_3}\Sel_{\phih}(\eabh) = - \omega(B) + \omega_1(A^3 - 27B).\]
		
		Note that for a typical $\eab$ of height $X$, both $B$ and $A^3-27B$ will have size $\asymp X$. By the Erd\H{o}s--Kac theorem, $\omega(B)$ will be normally distributed with mean and variance $\log\log X$, while $\omega_1(A^3 -27B)$ will be normally distributed with mean and variance $\frac12\log\log X$. Assuming independence, we therefore expect  
		\[\dim_{\F_3}\Sel_\phi(\eab) -\dim_{\F_3}\Sel_{\phih}(\eabh)\]
		to have mean $-\frac12\log\log X$ and variance $\frac32\log\log X$. Thus, for $E\in\E(X)$, we expect the Selmer ratio to be biased towards the dual isogeny.
		
		This heuristic argument has been made precise by \cite{Chan_2025}, who show, for example, that the random variable
		\[\frac{\dim_{\F_3}\Sel_\phi(\eab)-\dim_{\F_3}\Sel_{\phih}(\eabh) + \frac12\log\log X}{\sqrt{\frac32\log\log X}}\]
		for $E\in \E(X)$ converges to a standard Gaussian distribution as $X\to\infty$. They also show that the average size of $\dSelp$ is infinite. Thus, in this example, the global Selmer ratio dominates any other subtler phenomena.
		
		The goal of this paper is to investigate what happens to the distributions of $\dim_{\F_3}\Selp$ and $\dim_{\F_3}\dSelp$ when we control for the contribution of the Selmer ratio.
	\end{remark}
	
	\section{The cubic residue matrix}\label{sec:matrix}
	
	The goal of this section is to construct a matrix $M_{A, B}$ over $\F_3$, whose kernel is $\Sel_{\phih}(\eabh)$, viewed as a subspace of $\qmc$. A similar matrix was constructed in \cite{Chan_2023} for the elliptic curves $y^2 = x^3 + n^2$, by adapting results of \cite{Cohen_2009}, and our approach is similar. Throughout this section, we assume that $3\nmid AB$, or equivalently that $\eab$ has good reduction at $3$.
	
	\subsection{Local descent conditions}
	
	Each class in $\msc\Q$ has a unique cubefree integer representative $u=u_1^2u_2$, where $u_1, u_2$ are positive, squarefree, and coprime.
	
	\begin{theorem}[\cite{Cohen_2009}*{Thm.~3.1}]\label{thm:cohen}
		The class $[u_1^2u_2]\in\qmc$ belongs to $\Sel_{\phih}(\eabh)$ if and only if $u_1u_2\mid B$ and the plane cubic
		\begin{equation} \label{els}
			u_1 X^3 + u_2 Y^3 + \frac{B}{u_1u_2}Z^{3} - AXYZ=0
		\end{equation}
		is everywhere locally solvable. 
	\end{theorem}
	
	\begin{proof}
		Over a local field, \cite{Cohen_2009}*{Thm.~3.1} identifies the image of the Kummer map with the classes in which the corresponding plane cubic has a rational point. Applying this result at every place of $\Q$ gives the stated Selmer condition. Note that their model $y^2 = x^3 + (ax + b)^2$ is obtained from $\eab$ by completing the square, with $A = 2a$ and $B = 2b$.
	\end{proof}
	
	Set $u_3 = \frac B{u_1u_2}$. Permuting the three coefficients of \eqref{els} does not change the relevant Selmer condition. Indeed, we have
	\[[u_1^2u_2][B] = [u_2^2u_3],\qquad [u_1^2u_2]^2=[u_1u_2^2]\]
	in $\msc \Q$. Now assume that $(w_1, w_2, w_3)$ is a permutation of $(u_1,u_2,u_3)$. Since $[B]$, by \Cref{guaranteed_point_in_selmer}, already lies in the Selmer group, it follows that $[u_1^2u_2]\in \Sel_{\phih}(\eabh)$ if and only if $[w_1^2w_2]\in \Sel_{\phih}(\eabh)$.
	
	With our normalisation assumption that $p^3\nmid B$ whenever $p\mid A$, the local criteria in \cite{Cohen_2009}*{Sec.~5} reduce to conditions only at primes dividing $A^3 -27B$.
	
	\begin{theorem}[\cite{Cohen_2009}*{Sec.5}]
		\label{thm:els}
		Suppose that $u_1u_2\mid B$, that $3\nmid AB$. Then $[u_1^2u_2]$ is an element of $\Sel_{\phih}(\eabh)$ if and only if the following conditions hold:
		\begin{enumerate}
			\item If $p\mid A^3-27B$ and $p\nmid B$, then ${u_1}/{u_2}$ is a cube in $\F_p^{\times}$.
			\item If $p\mid A^3-27B$ and $p\mid B$, let $(w_1, w_2, w_3)$ be a permutation of $(u_1, u_2, u_3)$ with $v_p(w_1) \leq v_p(w_2) \leq v_p(w_3)$. Then either
			\begin{enumerate}
				\item $v_p(w_1) = v_p(w_2) = 0$ and $w_1 / w_2$ is a cube  in $\F_p^{\times}$; or
				\item $v_p(w_2) = v_p(w_3) = 1$ and $w_2 / w_3$ is a cube  in $\F_p^{\times}$.
			\end{enumerate}
		\end{enumerate}
	\end{theorem}
	
	\begin{proof}
		By \Cref{thm:cohen}, $[u_1^2u_2]\in \Sel_{\phih}(\eabh)$ if and only if the plane cubic
		\[u_1X^3+u_2Y^3+u_3Z^3-AXYZ=0\]
		is everywhere locally soluble. By \cite{Cohen_2009}*{Lemma~5.3}, a prime $p\ne 3$ with
		$p\nmid B(A^3-27B)$ gives no local condition. If $p\mid B$ but
		$p\nmid A^3-27B$, then $p\nmid A$, so \cite{Cohen_2009}*{Lemma~5.4(1)}
		again gives no local condition. Since $3\nmid A$, \cite{Cohen_2009}*{Lemma~5.6}
		gives no condition at $p=3$.
		
		It remains to consider primes $p\mid A^3-27B$. If $p\nmid B$, then
		\cite{Cohen_2009}*{Lemma~5.5} gives the condition that ${u_1}/{u_2}$ is a cube in $\F_p^{\times}$. If $p\mid B$, then necessarily $p\mid A$, and the stated criterion is
		exactly \cite{Cohen_2009}*{Lemma~5.4}, after permuting the three coefficients so that their $p$-adic valuations are ordered.
	\end{proof}
	
	
	\subsection{The restricted cubic residue symbol}
	
	In \Cref{thm:els}, we need to determine whether a value is a cube modulo $p$. One way to detect this would be to use the standard cubic residue symbol, however, this has the computational disadvantage of requiring us to work in $\Q(\zeta_3)$. Instead, we observe that we are only interested in determining when an integer $a$ is a cube modulo $p$. We can therefore define the following restricted symbol.
	
	\begin{definition}
		Let $p\ne 3$ be a prime and $a \in \Z$ such that $p \nmid a$. The cubic residue symbol $\left(\frac{a}{p}\right)_3$ is defined in $\F_p$ as follows:
		
		\begin{enumerate}
			\item If $p \equiv 2 \pmod 3$, then $\left(\frac{a}{p}\right)_3 = 1$.
			\item If $p \equiv 1 \pmod 3$, then $\left(\frac{a}{p}\right)_3 \equiv a^{\frac{p-1}{3}} \pmod p$. 
			This value lies in the set $\{1, \omega, \omega^2\}$, where $\omega$ is a primitive cube root of unity in $\F_p\t$.
		\end{enumerate}
	\end{definition}
	
	\begin{remark}
		The choice of $\omega, \omega^2\in \Fp\t$ is non-canonical. Indeed, in the usual definition of the cubic residue, this choice corresponds to choosing a prime ideal above $p$ in $\Z[\zeta_3]$. For consistency, we impose the convention that, viewed as elements of $\{1, \ldots, p-1\}$, we have $\omega<\omega^2$.
		
		If we normalise the standard cubic residue symbol with respect to this choice, our restricted cubic residue agrees with the standard residue when $a\iZ$.
	\end{remark}
	
	\begin{lemma}
		\label{cube_detection}
		An integer $a$ with $p\nmid a$ is a cubic residue modulo $p$ if and only if $\left(\frac{a}{p}\right)_3 = 1$.
	\end{lemma}
	\begin{proof}
		First of all, if $p \equiv 2 \pmod 3$, then $\gcd(3, p-1) = 1$ and by B\'ezout's Identity, there exist integers $u, v$ such that
		\( 3u + (p-1)v = 1 \).
		For any $a \in \F_p^\times$, we can write:
		\[ a =  a^{3u + (p-1)v} = (a^u)^3 \cdot (a^{p-1})^v \]
		By Fermat's little theorem, $a^{p-1} \equiv 1 \pmod p$, so
		\[ a \equiv (a^u)^3 \pmod p \]
		This proves that every element in $\F_p\t$ is a cube. 
		
		Let's now turn to the case where $p\equiv 1 \pmod 3$. The multiplicative group $\F_p^\times$ is cyclic of order $p-1$. Let $g$ be a primitive root modulo $p$ and suppose $a \equiv g^k \pmod p$ for some power $k$. We claim that $a$ is a cubic residue if and only if $3 \mid k$.
		\begin{itemize}
			\item $(\implies)$ If $a \equiv x^3 \pmod p$, then $a^{\frac{p-1}{3}} \equiv (x^3)^{\frac{p-1}{3}} \equiv x^{p-1} \equiv 1 \pmod p$.
			\item $(\impliedby)$ If $a^{\frac{p-1}{3}} \equiv 1 \pmod p$, we have
			\[ (g^k)^{\frac{p-1}{3}} \equiv g^{\frac{k(p-1)}{3}} \equiv 1 \pmod p \]
			Since $g$ is a primitive root, the exponent must be a multiple of the order $p-1$. That is, $p-1$ divides $\frac{k(p-1)}{3}$, which in turn implies that $3 \mid k$.
		\end{itemize}
		The roots of $X^3-1$ in $\F_p$ are $\{1, \omega, \omega^2\}$.
		If $3 \nmid k$, $a^{\frac{p-1}{3}}$ must take value in $\{\omega, \omega^2\}$ Thus, $\left(\frac{a}{p}\right)_3 = 1$ if and only if $a$ is a cube modulo $p$.
	\end{proof}
	
	\subsection{The matrix $M_{A,B}$}\label{subsec:matrix}
	
	We fix the following notation:
	\begin{itemize}[leftmargin=*]
		\item Let $p_1<p_2<\cdots < p_s$ be the primes congruent to $1\pmod 3$ that divide both $B$ and $A^3-27B$.
		\item Let $p_{s+1}<\cdots <p_n$ be the remaining primes dividing $B$.
		\item Set $v_j = v_{p_j}(B)$. For $j \le s$, we have $v_j\in \{1,2\}$.
		\item Put $q_i = p_i$ for $1\le i \le s$, and let $q_{s+1}<\cdots<q_\l$ be the primes congruent to $1\pmod 3$ that divide $A^3-27B$ but not $B$.
	\end{itemize}
	
	\begin{definition}
		\label{matrix}
		The cubic-residue matrix $\mab$ is the $\l\times n$ matrix over $\F_3$ given by
		\[
		(\mab)_{ij} = 
		\begin{cases}
			\displaystyle	 2 \log_\omega\left(\frac{B / {p_i}^{v_i}}{p_i}\right)_3 &\quad\text{ if }i \leq s \text{ and } i =  j \\
			\displaystyle\log_\omega \left(\frac{{p_j}^{v_i}}{q_i}\right)_3=\displaystyle\log_\omega \left(\frac{{p_j}^{v_i}}{p_i}\right)_3 &\quad\text{ if }  i\leq s \text{ and } i\neq j \\
			\displaystyle\log_\omega \left(\frac{p_j}{q_i}\right)_3 &\quad\text{ if }  s < i \leq l  \\
		\end{cases}
		\]
	\end{definition}
	
	Here $\log_{\omega}\:\{1,\omega,\omega^2\}\to\F_3$ maps $\omega^j\mapsto j$.
	
	\begin{theorem}
		\label{thm:kernel}
		The coordinate map
		\begin{align*}
			\F_3^n&\to\qmc\\
			(e_1, \ldots, e_n)&\mapsto \left[\prod_{j=1}^np_j^{e_j}\right]
		\end{align*}
		restricts to an isomorphism
		\[\ker\mab\xrightarrow{\sim}\Sel_{\hat\phi}(\eabh).\]
	\end{theorem}
	\begin{proof}
		By \Cref{thm:cohen}, every element $\Sel_{\phih}(\eabh)\sub\qmc$ has a unique cubefree representative $u = u_1^2u_2$ with $u_1u_2\mid B$. We can therefore write 
		\[u = \prod_{j = 1}^np_j^{e_j}\]
		for some vector $\mathbf{e} = (e_1, e_2 , \ldots, e_n) \in \F_3^n$.
		
		We will show that $u\in \Sel_{\phih}(\eabh)$ if and only if the corresponding vector $\mathbf{e}$ is in $\ker(M)$. We start by describing the vectors in $\ker(M)$. 
		
		When $1 \leq i  \leq s$, the $i$-th entry of $M\mathbf{e}$ is 
		\[
		(M\mathbf{e})_i=
		\begin{cases}
			\displaystyle \sum_{\substack{j=1\\j\ne i}}^{n} e_j \cdot  \log_\omega \left( \frac{{p_j}^{v_i}}{p_i} \right)_3 = v_i\log_{\omega} \left( \frac{\prod_{j = 1}^{n} {p_j}^{e_j}}{p_i} \right)_3  = v_i\log_\omega\left( \frac{u}{p_i} \right)_3&\text{ if } e_i = 0 \\
			\displaystyle\sum_{\substack{j=1\\j\ne i}}^{n} v_i\log_\omega \left( \frac{{p_j}^{e_j}}{p_i} \right)_3 +2  e_i\log_\omega\left(\frac{B/ {p_i}^{v_i}}{p_i}\right)_3 = 2\log_{\omega} \left[\left(\frac{B^{e_i} / u^{v_i}}{p_i}\right)_3\right] &\text{ if } e_i \neq  0 
		\end{cases}
		\]
		Let's see the implications of $(M\mathbf{e})_i = 0$ in terms of $u_1$, $u_2$, and $u_3$.
		\begin{enumerate}
			\item When $e_i = 0$, $(M\mathbf{e})_i = 0$ if and only if $\left( \frac{u}{p_i} \right)_3 = 1$, which is equivalent to $u$ being cube mod $p_i$. Since $u = u_1^2u_2$, $u$ being a cube is equivalent to ${u_1}/{u_2}$ being a cube. 
			\item When $e_i = 1$, $(M\mathbf{e})_i = 0$ if and only if $B/ u^{v_i}$ is a cube. If $v_i = 1$, since $u = u_1^2u_2$, this is equivalent to ${u_3}/{u_1}$ being a cube. If $v_i = 2$, this is equivalent to ${u_3}/{u_2}$ being a cube. 
			\item When $e_i = 2$, $(M\mathbf{e})_i = 0$ if and only if $B^2/ u^{v_i}$ is a cube. If $v_i = 1$, this is equivalent to ${u_2}/{u_3}$ being a cube. If $v_i = 2$, this is equivalent to ${u_1}/{u_3}$ being a cube. 
		\end{enumerate}
		In summary, if $1\le i \le s$, then
		\[
		(M\mathbf{e})_i = 0 \iff \begin{cases}
			u_1 / u_2 \text{ is a cube mod } p_i \qquad &\text{ if } e_i = 0 \\
			u_3 / u_1 \text{ is a cube mod } p_i & \text{ if } e_i = 1 \text{ and } v_i = 1\\
			u_3/u_2 \text{ is a cube mod } p_i &\text{ if } e_i = 1 \text{ and } v_i = 2 \\
			u_2/u_3 \text{ is a cube mod } p_i &\text{ if } e_i = 2 \text{ and } v_i = 1 \\
			u_1 / u_3 \text{ is a cube mod } p_i &\text{ if } e_i = 2 \text{ and } v_i = 2
		\end{cases}
		\]
		When $s+1 \leq i \leq \l$, we have
		\[
		(M\mathbf{e})_i = \sum_{j=1}^{n} e_j \cdot  \log_\omega \left( \frac{p_j}{q_i} \right)  = \log_\omega\left[\left( \frac{u}{q_i} \right)_3\right]
		\]
		Thus $(M\mathbf{e})_i = 0$ if and only if $u$ and hence $u_1 / u_2$ is a cube mod $q_i$.
		
		It remains to compare these conditions to \Cref{thm:els}. If $s + 1 \le i \le \l$, then $q_i\mid A^3-27B$ but $q_i\nmid B$, and the local condition at $q_i$ is satisfied if and only if $u_1/u_2$ is a cube mod $q_i$, which matches the condition $(M\mathbf{e})_i = 0$.
		
		So suppose that $1 \le i \le s$. Then $q_i = p_i$ divides both $B$ and $A^3-27B$. Since $u=u_1^2u_2$ is cubefree, the integers $u_1$ and $u_2$ are squarefree and coprime. Hence $v_{p_i}(u_1)$ and $v_{p_i}(u_2)$ are in $\{0,1\}$ and at most one of them is nonzero. Recall also that $v_i = v_{p_i}(B) \in\{1,2\}$. If $v_i = 1$, then the requirements in \Cref{thm:els} are:
		\begin{center}
			
			\begin{tabular}{|m {3.5em}|m {3.5em}|m {3.5em}|m {11em}|m {4.7em}|}		
				\hline 
				$v_p(u_1)$ & $v_p(u_2)$ & $v_p\left(u_3 \right)$ & Requirement &$ v_p(u) $\\
				\hline
				0 & 0 & 1 &  ${u_1}/{u_2}$ is a cube mod $p$ & 0 ($e_i = 0$)\\
				\hline
				0 & 1 & 0 & ${u_3}/{u_1} $ is a cube mod $p$ & 1 ($e_i = 1$)  \\
				\hline
				1 & 0 & 0 &  ${u_2}/{u_3}$ is a cube mod $p$ & 2 ($e_i = 2$)\\
				\hline
			\end{tabular}
		\end{center}
		
		If $v_i = 2$, then the requirements in \Cref{thm:els} are:
		\begin{center}
			\begin{tabular}{|m {3.5em}|m {3.5em}|m {3.5em}|m {11em}|m {4.7em}|}		
				\hline 
				$v_p(u_1)$ & $v_p(u_2)$ & $v_p\left(u_3 \right)$ & Requirement &$ v_p(u) $\\
				\hline
				0& 0  &2  &  ${u_1}/{u_2}$ is a cube mod $p$ & 0 ($e_i = 0$)\\
				\hline
				0 & 1 & 1 & $ {u_3}/{u_2} $ is a cube mod $p$ & 1 ($e_i = 1$) \\
				\hline
				1 & 0 & 1 &  ${u_1}/{u_3}$ is a cube mod $p$ & 2 ($e_i = 2$)\\
				\hline
			\end{tabular}
		\end{center}
		
		These tables match the previous summary: the local condition at $p_i$ in \Cref{thm:els} is satisfied if and only if $(M\mathbf{e})_i = 0$.
		
		Thus all the conditions in \Cref{thm:els} are satisfied if and only if $M\mathbf{e} = 0$, as required.
	\end{proof}
	
	\subsection{The reduced matrix $\mab'$}
	
	Write 
	\[B = \prod_{j= 1}^np_j^{v_j}\]
	By \Cref{guaranteed_point_in_selmer} and \Cref{thm:kernel}, the vector $(v_1, \ldots, v_n)$ lies in $\ker\mab$ and represents $[B]$. Recall that we have assumed that $B$ is not a rational cube, so this vector is non-zero. Let
	\[k = \max\{j : v_j \not\equiv 0\pmod 3\}.\]
	
	\begin{definition}
		The reduced cubic-residue matrix $M'_{A, B}$ is the matrix obtained from $\mab$ by deleting its $k$th column.
	\end{definition}
	
	\begin{theorem}\label{thm:ker'}
		There is an isomorphism
		\[\ker\mab'\simeq \Sel_{\hat\phi}(\eabh)/\langle [B]\rangle.\]
	\end{theorem}
	
	\begin{proof}
		Since the $k$-th coordinate of the kernel vector $(v_1,\ldots,v_n)$ is nonzero, deleting the $k$-th column does not change the rank of $M_{A,B}$. It follows from \Cref{thm:kernel} that
		\[ \operatorname{nullity}(M_{A, B})= \dim_{\F_3}\Sel_{\phih}(\eabh) \]
		and that
		\[ \operatorname{nullity}(M_{A, B}')= \dim_{\F_3}\Sel_{\phih}(\eabh) -1 = \dim_{\F_3}\br{\Sel_{\phih}(\eabh)/\langle [B]\rangle}.\]
	\end{proof}
	
	\subsection{The Selmer ratio and the dimension of $\Sel_\phi(\eab)$}
	
	Recall from \Cref{prop:selmer ratio} that the global Selmer ratio of $\eab$ is $3^{t_{A,B}}$, where
	\[t_{A,B} = -1 - \omega(B) + \omega_1(A^3-27B).\]
	
	\begin{theorem}\label{thm: matrix cols}
		Suppose that $M'_{A,B}$ has dimensions $\l\times m$. Then
		\[\l-m = 1-\omega(B)+\omega_1(A^3-27B)=t_{A,B}+2.\]
	\end{theorem}
	
	\begin{proof}
		By construction, the columns of $M_{A,B}$ are indexed by the distinct prime factors of $B$, and $M'_{A,B}$ has one column removed, so $m=\omega(B)-1$. The rows are indexed by the primes dividing $A^3-27B$ that are $1\pmod 3$, so $\l=\omega_1(A^3-27B)$. Thus
		\[
		\l-m=1-\omega(B)+\omega_1(A^3-27B).
		\]
		The equality with $t_{A,B}+2$ follows from \Cref{prop:selmer ratio}.
	\end{proof}
	\begin{corollary}
		\[
		\operatorname{nullity}((M'_{A,B})^{T}) = \dim_{\F_3} \sel_{\phi}(\eab)
		\]
	\end{corollary}
	\begin{proof}
		Let $r=\operatorname{rank}(M'_{A,B})$. By \Cref{thm: matrix cols}, we may assume that $M'_{A,B}$ has dimensions $(m + t_{A,B}+2)\times m$ for some $m$. By \Cref{thm:kernel}, we have
		\[
		\dim_{\F_3}\Sel_{\phih}(\eabh)=\operatorname{nullity}(M_{A,B}) = \operatorname{nullity}(M'_{A,B}) + 1=m+1-r.
		\]
		On the other hand, $\operatorname{rank}((M'_{A,B})^{T})=r$, so
		\[
		\operatorname{nullity}((M'_{A,B})^{T})=m + t_{A,B}+2 - r.
		\]
		By \Cref{thm:greenberg-wiles} and \Cref{prop:selmer ratio},
		\[
		\dim_{\F_3}\Sel_\phi(\eab)-\dim_{\F_3}\Sel_{\phih}(\eabh)
		=1+t_{A,B}
		\]
		Hence, 
		\begin{align*}
			\dim_{\F_3}\Sel_\phi(\eab) &= \dim_{\F_3}\Sel_{\phih}(\eabh)+1+t_{A,B}\\
			&= (m + 1 - r) + 1 + t_{A,B}\\
			&= m + t_{A, B} + 2 - r\\
			&= \operatorname{nullity}((M'_{A,B})^{T}).
		\end{align*}
	\end{proof}
	
	\section{The random matrix model}\label{sec:conjectures}
	
	Fix an elliptic curve $\eab$, and suppose that $c(\phi, \eab) = 3^{t-2}$. Set $m = \omega(B) - 1$. The results of \Cref{sec:matrix} give a $(m+t)\times m$ matrix $\mab'$ over $\F_3$ whose nullspace encodes $\dSelpd/\langle[ B]\rangle$. 
	Let $\E_t^{\mat}$ denote the subset of $\E_t$ on which $\mab'$ is defined: we require that
	\begin{enumerate}
		\item $B, A^3-27B\notin\Q^{\times 3}$;
		\item $3\nmid AB$.
	\end{enumerate}
	Let $\E_{t, m}^{\mat}$ be the subset of $\E_t^{\mat}$ consisting of curves $\eab$ with $\omega(B) -1 =m$, and let $\E_{t,m}^{\mat}(X)$ be the curves in $\E_{t,m}^{\mat}$ with height at most $X$.

	We posit that, as $\eab$ varies over $\E_{t,m}^\mat$, these matrices are random. More precisely, we propose the following conjecture:
	
	\begin{conjecture}\label{conj:componentwise}
		Fix integers $t$ and $m$ with $m\ge 0$ and $m+t\ge 0$. Then for every matrix $M\in \M_{(m+t)\times m}(\F_3)$,
		\[\lim_{X\to\infty}\frac{\#\set{\eab\in\E_{t,m}^{\mat}(X) : \mab' = M}}{\#\E_{t,m}^{\mat}(X)} = 3^{-m(m+t)}.\]
	\end{conjecture}
	
	The definition of $\mab'$ requires fixed conventions for the orderings of the primes that index its rows and columns. Thus, a priori, the distribution of the matrices should depend on these conventions. However, since the uniform measure on a matrix space is preserved by any change-of-basis matrix, we therefore expect  \Cref{conj:componentwise} should hold with any natural choice of these conventions.
	
	Nullity, on the other hand, is intrinsically basis-independent, and summing over all matrices of nullity $d$ gives
	\[\lim_{X\to\infty}\frac{\#\set{\eab\in\E_{t,m}^{\mat}(X) : \nullity(\mab') = d}}{\#\E_{t,m}^{\mat}(X)} = \P_{t, m}(d),\]
	where $\max(0,-t)\le d \le m$ and
	\[\P_{t, m}(d) = 3^{-m(m+t)}\prod_{i = 0}^{m-d-1}\frac{(3^{m+t}-3^i)(3^m-3^i)}{3^{m-d}-3^i}\]
	is the probability that a matrix $M\in \M_{(m+t)\times m}(\F_3)$ has nullity $d$ \cite{fulman-goldstein}*{Sec.~3}.
	
	\Cref{conj:componentwise} suggests an explicit random matrix model that matches the philosophy of the introduction: after controlling for the visible arithmetic constraints within the family, the remaining distribution should be governed by linear algebra. 
	
	Assuming that the convergence in \Cref{conj:componentwise} is sufficiently uniform, the conjecture leads naturally to the height-ordered prediction of \Cref{conj:intro}, at least within $\E_t^{\mat}$. Indeed, for every $d\geq\max(0,-t)$, we have
	\begin{align*}
		&\frac{
			\#\set{\eab\in\E_t^{\mat}(X):
				\nullity(\mab')=d}
		}{
			\#\E_t^{\mat}(X)
		}\\
		&=
		\sum_{\substack{m\geq0\\m+t\geq0}}
		\frac{\#\E_{t,m}^{\mat}(X)}
		{\#\E_t^{\mat}(X)}\cdot
		\frac{
			\#\set{\eab\in\E_{t,m}^{\mat}(X):
				\nullity(\mab')=d}
		}{
			\#\E_{t,m}^{\mat}(X)
		}.
	\end{align*}
	For each fixed $m$, \Cref{conj:componentwise} predicts that the second factor tends to $\P_{t,m}(d)$. On the other hand, one expects $m=\omega(B)-1$ to tend to infinity in probability in the height-ordered family, even after conditioning on the global Selmer ratio. Since $\lim_{m\to\infty}\P_{t,m}(d)=\P_t(d)$, uniformity over the values of $m$ carrying most of the height-ordered mass would give
	\[
	\lim_{X\to\infty}
	\frac{
		\#\set{\eab\in\E_t^{\mat}(X):
			\nullity(\mab')=d}
	}{
		\#\E_t^{\mat}(X)
	}
	=\P_t(d).
	\]
	Since
	\[
	\nullity(\mab')
	=
	\dim_{\F_3}\dSelp-1,
	\]
	this is precisely the distribution asserted in \Cref{conj:intro}, restricted to $\E_t^{\mat}$. 
	
	It remains to compare $\E_t^{\mat}$ with the full family $\E_t$. There are two kinds of curves omitted from $\E_t^{\mat}$.  First, we remove the curves for which $B$ or $A^3-27B$ is a rational cube. These are the thin subfamilies with additional $3$-power isogenies considered in \Cref{lem:extra-isogenies}, and as discussed there, we expect them to be negligible in $\E_t$. Their additional isogenies introduce genuinely different arithmetic structure, so they should be regarded as exceptional subfamilies in the same way that $\E$ is an exceptional subfamily of the family of all elliptic curves over $\Q$.
	
	The second restriction, $3\nmid AB$, is different: we expect that curves with bad reduction at $3$ form a positive proportion of curves in $\E_t$. The reason we nevertheless expect them to have the same limiting distribution is that the obstruction is confined to the local condition at one fixed prime. Indeed, the local Selmer condition at $3$ lies in a cohomology group of bounded dimension, and thus should contribute a fixed small number of rows and columns to the matrix. Moreover, it is reasonable to expect that, after conditioning on the Selmer ratio, the image of the $3$-adic local Kummer map into $H^1(\Q_3,\mu_3)\cong\msc{\Q_3}$ is asymptotically independent of the residue-symbol conditions coming from the other primes. 
	The difference between the numbers of additional rows and columns is already encoded by the Selmer ratio. Thus, under this independence heuristic, the local condition at $3$ should amount to adjoining only a bounded number of random rows and columns to the matrix. Since the number of prime factors of $B$ and $A^3-27B$, and therefore the dimensions of $\mab$, typically grow with the height of $\eab$, the columns and rows coming from the bad reduction at $3$ should not change the limiting distribution.

	\section{Computational Experiments}\label{sec:results}
	
	\subsection{The Experiment Setup}
	
	Our experiments test the distributions of the matrices $\mab'$ of fixed dimensions, as proposed in \Cref{conj:componentwise}. If $\mab'$ has dimensions $a\times b$, then 
	\[b = \omega(B) - 1\qquad a = b+t, \]
	so fixing the dimensions of $\mab'$ is equivalent to fixing the parameter $t$ and the number of prime factors of $B$.
	
	Our first experiment was to compute the matrices $\mab'$ for all elliptic curves $\eab$ up to some given height $X$. However, this approach yielded matrices that were \emph{not} equidistributed, a fact that, as we explained in the introduction, was to be expected in hindsight. Indeed, at the heights we can reach computationally, fixing the matrix dimensions selects values of $B$ with unusually many small prime factors. We instead use two different sampling methods to model the behaviour for large heights. The first retains a height ordering, but sieves out elliptic curves whose discriminants have small prime factors. The second models the Swinnerton-Dyer ordering \cite{swinnerton-dyer-heuristics} by fixing $\omega(B)$ and drawing primes from a large pool. 
	
	\subsubsection{Sifted height method}
	
	Fix integers $h_0<h_1$ and a prime cutoff $K>3$. We compute $\mab'$ for every normalised pair $(A, B)$ satisfying:
	\begin{itemize}
		\item $h_0\le H(\eab) \le h_1$;
		\item no prime $p<K$ divides $B(A^3-27B)$;
		\item neither $B$ nor $A^3-27B$ is a cube.
	\end{itemize}
	
	The sieve removes the small primes most likely to dominate the cubic residue symbols, while the restriction to a height interval allows us to work at larger heights without first enumerating every smaller curve. This method exhausts all elliptic curves satisfying the given restrictions. However, due to the sieve, it is rare for $B$ and $A^3-27B$ to have many prime factors, so these experiments produce mostly small matrices.

	\subsubsection{Prime factor method}
	
	To study larger matrices, we use a second method inspired by the prime factor ordering of Swinnerton-Dyer in quadratic twist families \cite{swinnerton-dyer-heuristics}. Rather than ordering the twisting parameter by its size, Swinnerton-Dyer instead fixes its number of prime factors and allows those primes to vary. In Swinnerton-Dyer's setting, the twisting parameter determines the local Selmer conditions in the family. In our setting, that role is played by the integers $B$ and $A^3-27B$. In our experiment, we follow a similar idea by fixing $\omega(B)$ and choosing the prime divisors of $B$ from a large finite pool.
	
	Fix positive integers $N$ and $n$, and let $\mathcal{P}_N$ be the set of the first $N$ primes greater than $3$. We choose $n$ distinct primes $p_1, \ldots, p_n\in \mathcal P_N$ uniformly without replacement. For each $p_i$, independently, we choose a positive exponent $e_i$ with probability
	\[\Prob(e_i = k) = \frac{1}{p_i^{k-1}}-\frac 1{p_i^k}\]
	for each $k\ge 1$, i.e.\ the probability that a random integer $d$ is divisible by $p_i^{k}$ but not $p_i^{k+1}$ given that it is divisible by $p_i$. We then set
	\[B = \prod_{i = 1}^np_i^{e_i}.\]
	If $B$ is a cube, we discard it and restart.
	
	Next, we select $|A|$ uniformly so that it is not divisible by $3$ and its height is approximately that of $B$, i.e.
	\[0.9\sqrt[3]B \le |A| < 1.1\sqrt[3]B,\]
	and we choose the sign of $A$ at random. If the resulting pair $(A, B)$ would not be normalised, we choose a new $A$.
	If $A^3-27B$ is a cube, we discard the pair and restart.
	
	\begin{remark}\label{rem:ast}
		When $n = 2$, the interval $(0.9\sqrt[3]B, 1.1\sqrt[3]B)$ frequently contains no integers that are coprime to $3$. Thus, when $n = 2$, we instead used the entire range $A\in (-1.1\sqrt[3]B, 1.1\sqrt[3]B)$. These experiments are marked with an asterisk.
	\end{remark}
	
	Since $\omega(B) = n$, the reduced matrix $\mab'$ has $n-1$ columns, while its number of rows can vary. We only retain matrices with at most $12$ entries. This cutoff is computational: there are $3^{ab}$ possible $a\times b$ matrices over $\F_3$, so gathering enough data to test equidistribution is impractical when $ab$ is too large.

	\subsection{Statistical tests}
	
	For each experiment, we organise our results into blocks according to the dimensions $a\times b$ of $\mab'$. Within a fixed block, let $\Count(M)$ be the number of occurrences of $M\in\M_{a\times b}(\F_3)$, and let 
	\[\Total = \sum_{M\in\M_{a\times b}(\F_3)}\Count(M)\]
	be the total number of matrices in the block. There are $3^{ab}$ possible matrices and the expected count is $\Total/3^{ab}$. 
	
	We record the observed standard deviation
	\[\stdev_{\obs} =\sqrt{\frac{1}{3^{ab}}\sum_{M\in\M_{a\times b}(\F_3)}\br{\Count(M) - \frac{\Total}{3^{ab}}}^2}\]
	and compare it with the standard deviation of the uniform multinomial distribution
	\[\stdev_{\unif} = \sqrt{\frac{\Total}{3^{ab}}\br{1-\frac1{3^{ab}}}}.\]
	We call $\frac{\stdev_{\obs}}{\stdev_{\unif}}$ the standard deviation ratio. For independent uniform draws, its square has expectation $1$.
	
	We also record the mean squared error of the observed probabilities
	\[\MSE = \frac{1}{3^{ab}}\sum_{M\in\M_{a\times b}(\F_3)}\br{\frac{\Count(M)}{\Total}-\frac{1}{3^{ab}}}^2,\]
	the minimum and maximum counts, and their relative deviations from $\Total/3^{ab}$.
	
	\begin{remark}
		We have
		\[\MSE = \frac{\stdev_{\obs}^2}{\Total^2}.\]
		and
		\[\br{\frac{\stdev_{\obs}}{\stdev_{\unif}}}^2
		=\frac{\Total\,3^{2ab}}{3^{ab}-1}\MSE.\]
		The mean squared error measures the deviation in probability, while the standard deviation error also depends on $\Total$ and $3^{ab}$. Therefore the standard deviation ratio cannot be used to compare experiments with very different values of $\Total$ and $3^{ab}$. 
	\end{remark}
	
	Finally, we record the maximal entry discrepancy,
	\[\max_{i,j,u}\abs{\Prob_{\obs}(M_{ij} = u) -\frac13},\]
	the largest bias visible in any single matrix entry; and the maximal rank discrepancy
	\[\max_r\abs{\Prob_{\obs}(\rk M = r) - \Prob_{\unif}(\rk M = r)},\]
	which measures the extent to which any deviation from uniformity affects the rank distribution.
	
	Complete statistics for each experiment are given in \Cref{appendix}. The raw data is available at our at our \href{https://github.com/dcz0711/3-Selmer-Groups/}{Github repository}.
	
	\begin{remark}
		We do not use Pearson's $\chi^2$ test, which is the standard goodness-of-fit test for the exact uniform distribution. Its statistic is given by
		\[\chi^2 = \sum_{M\in\M_{a\times b}(\F_3)}\frac{(\Count(M)-\frac{\Total}{3^{ab}})^2}{\frac{\Total}{3^{ab}}}.\]
		Set
		\[p_M(X) = \Prob_{\eab\in\E_{t,m}^{\mat}(X)}( \mab' = M)\]
		and let
		\[\epsilon(X) = \max_{M\in\M_{(m+t)\times m}(\F_3)}\abs{p_M(X) - 3^{-m(m+t)}}.\]
		\Cref{conj:componentwise} asserts only that $\epsilon(X)\to 0$ as $X\to\infty$. Now suppose that $\Total$ matrices are drawn independently from a distribution satisfying $\epsilon(X)\to 0$. Then
		\[\EE[\chi^2] = (3^{m(m+t)} - 1)+(\Total - 1) \br{\sum_{M\in\M_{(m+t)\times m}(\F_3)}\frac{\br{p_M(X)-3^{-m(m+t)}}^2}{3^{-m(m+t)}}} \asymp_{m,t} 1+ \Total\epsilon(X)^2. \]
		Hence, if $\Total\epsilon(X)^2\to\infty$, then the $\chi^2$ statistic will diverge even though $\epsilon(X)\to 0$.
		In the setting of \Cref{conj:componentwise}, there is no reason to expect that $\epsilon(X) = O(\Total^{-1/2})$. As noted in the introduction, the analogous work of Kane--Klagsbrun contains error terms that decay only as powers of $\log\log X$ \cite{kane-klagsbrun}. The same objection applies to other goodness-of-fit tests. We therefore report the sizes of the observed discrepancies rather than formal $p$-values.
	\end{remark}

	\subsection{Analysis of Results}
	
	\subsubsection{Sifted height experiments}
	
	The parameters and total numbers of curves in the sifted height experiments
	are given in \Cref{tab:height-experiments}. A matrix $\mab'$ is trivial if one of its rows or columns is zero-dimensional.
	
	\begin{table}[!h]
		\centering
		\caption{Parameters and sample sizes for the sifted height experiments.}
		\label{tab:height-experiments}
		\scriptsize
		\resizebox{\textwidth}{!}{%
			\begin{tabular}{llrrr}
				\toprule
				ID & Parameters & Total generated & Trivial & Non-trivial \\
				\midrule
				H-1000-100 & $1000^3\le H(\eab)\le1010^3$, good reduction at all primes $<100$ & $1{,}244{,}605{,}158$ & $743{,}432{,}900$ & $501{,}172{,}258$ \\
				H-1000-300 & $1000^3\le H(\eab)\le1010^3$, good reduction at all primes $<300$ & $816{,}536{,}769$ & $552{,}408{,}952$ & $264{,}127{,}817$ \\
				H-1000-500 & $1000^3\le H(\eab)\le1010^3$, good reduction at all primes $<500$ & $690{,}451{,}279$ & $491{,}069{,}894$ & $199{,}381{,}385$ \\
				H-1000-800 & $1000^3\le H(\eab)\le1010^3$, good reduction at all primes $<800$ & $603{,}045{,}814$ & $445{,}919{,}536$ & $157{,}126{,}278$ \\
				H-1200 & $1200^3\le H(\eab)\le1300^3$, good reduction at all primes $<500$ & $12{,}638{,}762{,}546$ & $8{,}816{,}993{,}327$ & $3{,}821{,}769{,}219$ \\
				H-2250 & $2250^3\le H(\eab)\le2260^3$, good reduction at all primes $<2000$ & $5{,}267{,}460{,}247$ & $3{,}928{,}485{,}994$ & $1{,}338{,}974{,}253$ \\
				H-3250 & $3250^3\le H(\eab)\le3260^3$, good reduction at all primes $<1000$ & $19{,}442{,}965{,}506$ & $13{,}390{,}100{,}389$ & $6{,}052{,}865{,}117$ \\
				\bottomrule
		\end{tabular}}
	\end{table}
	
	Experiment H-3250 is representative of the sifted height data as a whole. It exhausts all elliptic curves $\eab$ such that
	\[3250^3\le H(\eab) \le 3260^3\]
	for which no prime below $1000$ divides $B(A^3-27B)$.
	
	\begin{table}[htbp]
		\centering
		\caption{Statistics for the H-3250 experiment.}
		\label{tab:H3250-statistics}
		\scriptsize
		\resizebox{\textwidth}{!}{%
			\begin{tabular}{crrrrrrrrrrrrr}
				\toprule
				\makecell[c]{Dimensions}
				& \makecell[c]{Total generated}
				& \makecell[c]{Possible matrices}
				& \makecell[c]{Mean count}
				& \makecell[c]{Observed\\std.\ dev.}
				& \makecell[c]{Uniform\\std.\ dev.}
				& \makecell[c]{SD\\ratio}
				& \makecell[c]{Minimum}
				& \makecell[c]{\% dev.}
				& \makecell[c]{Maximum}
				& \makecell[c]{\% dev.}
				& \makecell[c]{$\MSE$}
				& \makecell[c]{Max. entry\\discrepancy}
				& \makecell[c]{Max. rank\\discrepancy} \\
				\midrule
				$1\times1$
				& $4{,}427{,}036{,}888$
				& $3$
				& $1{,}475{,}678{,}962.67$
				& $575{,}557.64$
				& $31{,}365.36$
				& $18.350$
				& $1{,}474{,}865{,}895$
				& $-0.06\%$
				& $1{,}476{,}118{,}521$
				& $+0.03\%$
				& $1.69\mathbin{\cdot}10^{-8}$
				& $1.84\mathbin{\cdot}10^{-4}$
				& $1.84\mathbin{\cdot}10^{-4}$ \\
				
				$2\times1$
				& $1{,}292{,}105{,}782$
				& $9$
				& $143{,}567{,}309.11$
				& $194{,}822.13$
				& $11{,}296.70$
				& $17.246$
				& $143{,}283{,}441$
				& $-0.20\%$
				& $143{,}716{,}189$
				& $+0.10\%$
				& $2.27\mathbin{\cdot}10^{-8}$
				& $6.39\mathbin{\cdot}10^{-4}$
				& $2.17\mathbin{\cdot}10^{-4}$ \\
				
				$3\times1$
				& $66{,}615{,}066$
				& $27$
				& $2{,}467{,}224.67$
				& $8{,}144.83$
				& $1{,}541.38$
				& $5.284$
				& $2{,}450{,}982$
				& $-0.66\%$
				& $2{,}480{,}888$
				& $+0.55\%$
				& $1.49\mathbin{\cdot}10^{-8}$
				& $1.12\mathbin{\cdot}10^{-3}$
				& $1.51\mathbin{\cdot}10^{-4}$ \\
				
				$1\times2$
				& $203{,}960{,}440$
				& $9$
				& $22{,}662{,}271.11$
				& $26{,}159.53$
				& $4{,}488.23$
				& $5.828$
				& $22{,}615{,}982$
				& $-0.20\%$
				& $22{,}704{,}317$
				& $+0.19\%$
				& $1.65\mathbin{\cdot}10^{-8}$
				& $3.61\mathbin{\cdot}10^{-4}$
				& $2.27\mathbin{\cdot}10^{-4}$ \\
				
				$2\times2$
				& $59{,}973{,}424$
				& $81$
				& $740{,}412.64$
				& $3{,}300.93$
				& $855.14$
				& $3.860$
				& $731{,}360$
				& $-1.22\%$
				& $746{,}325$
				& $+0.80\%$
				& $3.03\mathbin{\cdot}10^{-9}$
				& $1.295\mathbin{\cdot}10^{-3}$
				& $5.85\mathbin{\cdot}10^{-4}$ \\
				
				$3\times2$
				& $3{,}173{,}517$
				& $729$
				& $4{,}353.25$
				& $82.17$
				& $65.93$
				& $1.246$
				& $4{,}074$
				& $-6.41\%$
				& $4{,}597$
				& $+5.60\%$
				& $6.70\mathbin{\cdot}10^{-10}$
				& $3.062\mathbin{\cdot}10^{-3}$
				& $5.86\mathbin{\cdot}10^{-4}$ \\
				\bottomrule
		\end{tabular}}
	\end{table}
	
	In terms of observed probabilities, the data is extremely close to uniform. In the three largest blocks, every individual matrix occurs within $0.21\%$ of its uniform prediction, and the maximum rank discrepancies are all within $6\times 10^{-4}$. At the same time, the large standard deviation ratios are much larger than $1$, so the counts do not have the same fluctuations as independent uniform draws. With hundreds of millions of matrices, a very small departure from uniformity is still easy to detect. Experiment H-3250 is consistent with convergence to the uniform distribution, but the heights we can reach are not yet large enough for the matrices to behave truly independently. 
	
	The data also supports the claim made in the introduction that, at the heights we can reach computationally, small primes have a disproportionate influence on the distribution. Because the four H-1000 experiments use the same height interval, they form a nested sequence, and thus allow us to analyse which curves contribute most to the failure of equidistribution. For example, by subtracting the counts of experiment H-1000-300 from the counts of experiment H-1000-100, we can compare the curves with bad reduction at some prime $100\le p<300$ with those with good reduction at all primes $p<300$. After stratifying by the smallest prime of bad reduction, we summarise the data for these experiments in \Cref{tab:height-sieve-comparison}. Complete statistics for all the sifted height experiments are given in \Cref{tab:full-height-statistics}.
	
	In each of the five largest blocks, the matrices with $100\le p_{\min}<300$ are substantially less uniform than those with $300\le p_{\min}<500$, which are in turn less uniform than those with $500\le p_{\min}<800$. This improvement appears simultaneously in nearly all our metrics. In contrast, the $p_{\min}\ge 800$ range shows no consistent further improvement in the well-populated blocks. Thus, the failure of equidistribution is not uniform across the entire height range: it is concentrated disproportionately among the curves with small primes of bad reduction.
	
	The results of the remaining experiments, H-1200 and H-2250, are consistent with this overall picture. Their complete statistics are recorded in \Cref{tab:full-height-statistics}.
	
	\subsubsection{Prime factor experiments}
	
	The prime factor experiments are listed in
	\Cref{tab:factor-experiments}. We call a matrix trivial if one of its dimensions is zero. The asterisks on the experiments with $n=2$
	are explained in \Cref{rem:ast}. 
	
	\begin{table}[!h]
		
		\caption{Distribution statistics for experiment F10k-7.}
		\label{tab:F10k7-statistics}
		\scriptsize
		\resizebox{\textwidth}{!}{%
			\begin{tabular}{crrrrrrrrrrrrr}
				\toprule
				\makecell[c]{Dimensions}
				& \makecell[c]{Total generated}
				& \makecell[c]{Possible matrices}
				& \makecell[c]{Mean count}
				& \makecell[c]{Observed\\std.\ dev.}
				& \makecell[c]{Uniform\\std.\ dev.}
				& \makecell[c]{SD\\ratio}
				& \makecell[c]{Minimum}
				& \makecell[c]{\% dev.}
				& \makecell[c]{Maximum}
				& \makecell[c]{\% dev.}
				& \makecell[c]{$\MSE$}
				& \makecell[c]{Max. entry\\discrepancy}
				& \makecell[c]{Max. rank\\discrepancy} \\
				\midrule
				$1\times6$
				& $51{,}722{,}830$
				& $729$
				& $70{,}950.38$
				& $291.67$
				& $266.18$
				& $1.096$
				& $70{,}104$
				& $-1.19\%$
				& $71{,}706$
				& $+1.06\%$
				& $3.18\mathbin{\cdot}10^{-11}$
				& $6.133\mathbin{\cdot}10^{-4}$
				& $6.832\mathbin{\cdot}10^{-6}$ \\
				
				$2\times6$
				& $52{,}656{,}296$
				& $531{,}441$
				& $99.08$
				& $9.97$
				& $9.95$
				& $1.002$
				& $58$
				& $-41.46\%$
				& $150$
				& $+51.39\%$
				& $3.59\mathbin{\cdot}10^{-14}$
				& $1.075\mathbin{\cdot}10^{-3}$
				& $2.283\mathbin{\cdot}10^{-5}$ \\
				\bottomrule
		\end{tabular}}
	\end{table}

	\begin{table}[b]
		\caption{Distribution statistics for the H-1000 samples stratified by the smallest prime divisor $p_{\min}$ of $B(A^3-27B)$. The minimum and
			maximum deviations are measured relative to the uniform mean
			$\Total/3^{ab}$.}
		\label{tab:height-sieve-comparison}
		\centering
		\scriptsize
		\resizebox{\textwidth}{!}{%
			\begin{tabular}{ccrrrrrrr}
				\toprule
				\makecell[c]{\scriptsize Condition on\\[-1pt]\scriptsize $p_{\min}$}
				& \makecell[c]{\scriptsize Dimensions}
				& \makecell[c]{\scriptsize Total\\[-1pt]\scriptsize generated}
				& \makecell[c]{\scriptsize Min.\ \% dev.}
				& \makecell[c]{\scriptsize Max.\ \% dev.}
				& \makecell[c]{\scriptsize $\stdev$ ratio}
				& \makecell[c]{\scriptsize $\MSE$}
				& \makecell[c]{\scriptsize Max.\ entry\\[-1pt]\scriptsize discrepancy}
				& \makecell[c]{\scriptsize Max.\ rank\\[-1pt]\scriptsize discrepancy} \\
				\midrule

				$100\leq p_{\min}<300$
				& $1\times1$ & $108{,}239{,}046$ & $-1.00\%$ & $+0.58\%$
				& $52.043$ & $5.56\cdot10^{-6}$ & $3.319\cdot10^{-3}$ & $3.319\cdot10^{-3}$ \\
				& $2\times1$ & $52{,}345{,}364$ & $-2.87\%$ & $+1.58\%$
				& $44.604$ & $3.75\cdot10^{-6}$ & $8.083\cdot10^{-3}$ & $3.187\cdot10^{-3}$ \\
				& $3\times1$ & $8{,}211{,}532$ & $-5.41\%$ & $+2.94\%$
				& $13.623$ & $8.06\cdot10^{-7}$ & $9.935\cdot10^{-3}$ & $2.005\cdot10^{-3}$ \\
				& $1\times2$ & $47{,}060{,}163$ & $-2.02\%$ & $+1.16\%$
				& $24.628$ & $1.27\cdot10^{-6}$ & $4.031\cdot10^{-3}$ & $2.244\cdot10^{-3}$ \\
				& $2\times2$ & $18{,}229{,}199$ & $-7.34\%$ & $+4.15\%$
				& $14.731$ & $1.45\cdot10^{-7}$ & $1.205\cdot10^{-2}$ & $5.148\cdot10^{-3}$ \\
				& $3\times2$ & $2{,}116{,}260$ & $-18.67\%$ & $+11.89\%$
				& $3.155$ & $6.44\cdot10^{-9}$ & $1.940\cdot10^{-2}$ & $5.339\cdot10^{-3}$ \\
				\addlinespace

				$300\leq p_{\min}<500$
				& $1\times1$ & $36{,}093{,}824$ & $-0.61\%$ & $+0.39\%$
				& $18.473$ & $2.10\cdot10^{-6}$ & $2.024\cdot10^{-3}$ & $2.024\cdot10^{-3}$ \\
				& $2\times1$ & $15{,}663{,}950$ & $-1.68\%$ & $+1.06\%$
				& $15.827$ & $1.58\cdot10^{-6}$ & $5.281\cdot10^{-3}$ & $1.869\cdot10^{-3}$ \\
				& $3\times1$ & $1{,}839{,}393$ & $-3.27\%$ & $+2.27\%$
				& $4.458$ & $3.85\cdot10^{-7}$ & $6.653\cdot10^{-3}$ & $1.192\cdot10^{-3}$ \\
				& $1\times2$ & $8{,}301{,}120$ & $-1.01\%$ & $+0.67\%$
				& $5.578$ & $3.70\cdot10^{-7}$ & $2.102\cdot10^{-3}$ & $1.122\cdot10^{-3}$ \\
				& $2\times2$ & $2{,}659{,}052$ & $-5.93\%$ & $+3.65\%$
				& $4.104$ & $7.72\cdot10^{-8}$ & $8.434\cdot10^{-3}$ & $3.725\cdot10^{-3}$ \\
				& $3\times2$ & $187{,}620$ & $-20.35\%$ & $+25.50\%$
				& $1.229$ & $1.10\cdot10^{-8}$ & $1.709\cdot10^{-2}$ & $3.046\cdot10^{-3}$ \\
				\addlinespace

				$500\leq p_{\min}<800$
				& $1\times1$ & $27{,}133{,}678$ & $-0.15\%$ & $+0.20\%$
				& $5.327$ & $2.32\cdot10^{-7}$ & $6.502\cdot10^{-4}$ & $5.025\cdot10^{-4}$ \\
				& $2\times1$ & $10{,}395{,}320$ & $-0.47\%$ & $+0.67\%$
				& $4.817$ & $2.20\cdot10^{-7}$ & $1.914\cdot10^{-3}$ & $3.792\cdot10^{-4}$ \\
				& $3\times1$ & $940{,}791$ & $-1.67\%$ & $+1.61\%$
				& $1.922$ & $1.40\cdot10^{-7}$ & $3.113\cdot10^{-3}$ & $3.105\cdot10^{-4}$ \\
				& $1\times2$ & $2{,}924{,}964$ & $-0.51\%$ & $+0.52\%$
				& $1.835$ & $1.14\cdot10^{-7}$ & $9.210\cdot10^{-4}$ & $5.631\cdot10^{-4}$ \\
				& $2\times2$ & $823{,}226$ & $-4.03\%$ & $+3.99\%$
				& $1.603$ & $3.81\cdot10^{-8}$ & $4.064\cdot10^{-3}$ & $2.446\cdot10^{-3}$ \\
				& $3\times2$ & $37{,}128$ & $-33.24\%$ & $+59.04\%$
				& $1.005$ & $3.73\cdot10^{-8}$ & $1.2336\cdot10^{-2}$ & $2.819\cdot10^{-3}$ \\
				\addlinespace
				$p_{\min}\geq800$
				& $1\times1$ & $124{,}309{,}856$ & $-0.11\%$ & $+0.08\%$
				& $6.260$ & $7.01\cdot10^{-8}$ & $3.64\cdot10^{-4}$ & $3.64\cdot10^{-4}$ \\
				& $2\times1$ & $31{,}604{,}006$ & $-0.50\%$ & $+0.41\%$
				& $6.739$ & $1.42\cdot10^{-7}$ & $1.512\cdot10^{-3}$ & $3.93\cdot10^{-4}$ \\
				& $3\times1$ & $825{,}049$ & $-2.14\%$ & $+2.07\%$
				& $1.904$ & $1.57\cdot10^{-7}$ & $2.678\cdot10^{-3}$ & $3.11\cdot10^{-4}$ \\
				& $1\times2$ & $306{,}603$ & $-0.96\%$ & $+1.36\%$
				& $1.402$ & $6.33\cdot10^{-7}$ & $2.635\cdot10^{-3}$ & $7.37\cdot10^{-4}$ \\
				& $2\times2$ & $78{,}666$ & $-8.87\%$ & $+9.35\%$
				& $1.113$ & $1.92\cdot10^{-7}$ & $7.589\cdot10^{-3}$ & $4.565\cdot10^{-3}$ \\
				& $3\times2$ & $2{,}098$ & $-100.00\%$ & $+247.47\%$
				& $1.010$ & $6.66\cdot10^{-7}$ & $1.7795\cdot10^{-2}$ & $7.713\cdot10^{-3}$ \\
				\bottomrule
		\end{tabular}}
	\end{table}
	
	The factor method allows us to test substantially larger matrix spaces than the sifted height method. Experiment F10k-7 is informative. It sampled $124{,}800{,}000$ elliptic curves $\eab$ with $\omega(B) = 7$, and with the prime factors of $B$ drawn from the first $10{,}000$ primes $>3$. There were $104{,}379{,}126$ non-trivial matrices, separated into two blocks. The data is shown in \Cref{tab:F10k7-statistics}.

	\begin{table}[t]
		\vspace{-30pt}
		\centering
		\caption{Parameters and sample sizes for the prime factor experiments.}
		\label{tab:factor-experiments}
		\scriptsize
		\begin{tabular*}{\textwidth}{@{\extracolsep{\fill}}lrrrrr@{}}
			\toprule
			\multicolumn{1}{c}{ID}
			& \multicolumn{1}{c}{$N$}
			& \multicolumn{1}{c}{$n$}
			& \multicolumn{1}{c}{Total generated}
			& \multicolumn{1}{c}{Trivial}
			& \multicolumn{1}{c}{Non-trivial} \\
			\midrule
			F1k-2$^\ast$ & $1{,}000$ & $2$ & $100{,}000{,}000$ & $24{,}349{,}718$ & $75{,}650{,}282$ \\
			F10k-2$^\ast$ & $10{,}000$ & $2$ & $100{,}000{,}000$ & $21{,}522{,}121$ & $78{,}477{,}879$ \\
			F100k-2$^\ast$ & $100{,}000$ & $2$ & $100{,}000{,}000$ & $19{,}558{,}563$ & $80{,}441{,}437$ \\
			\addlinespace
			F1k-3 & $1{,}000$ & $3$ & $100{,}000{,}000$ & $20{,}452{,}903$ & $79{,}547{,}097$ \\
			F10k-3 & $10{,}000$ & $3$ & $100{,}000{,}000$ & $17{,}982{,}895$ & $82{,}017{,}105$ \\
			F100k-3 & $100{,}000$ & $3$ & $100{,}000{,}000$ & $16{,}269{,}254$ & $83{,}730{,}746$ \\
			\addlinespace
			F1k-4 & $1{,}000$ & $4$ & $100{,}000{,}000$ & $18{,}038{,}420$ & $81{,}961{,}580$ \\
			F10k-4 & $10{,}000$ & $4$ & $110{,}000{,}000$ & $17{,}444{,}654$ & $92{,}555{,}346$ \\
			F100k-4 & $100{,}000$ & $4$ & $100{,}000{,}000$ & $14{,}382{,}324$ & $85{,}617{,}676$ \\
			\addlinespace
			F1k-5 & $1{,}000$ & $5$ & $100{,}000{,}000$ & $16{,}902{,}483$ & $83{,}097{,}517$ \\
			F10k-5 & $10{,}000$ & $5$ & $100{,}780{,}000$ & $15{,}170{,}957$ & $85{,}609{,}043$ \\
			F100k-5 & $100{,}000$ & $5$ & $100{,}000{,}000$ & $13{,}807{,}435$ & $86{,}192{,}565$ \\
			\addlinespace
			F1k-6 & $1{,}000$ & $6$ & $100{,}000{,}000$ & $18{,}560{,}765$ & $81{,}439{,}235$ \\
			F10k-6 & $10{,}000$ & $6$ & $100{,}000{,}000$ & $17{,}085{,}873$ & $82{,}914{,}127$ \\
			F100k-6 & $100{,}000$ & $6$ & $100{,}000{,}000$ & $16{,}067{,}177$ & $83{,}932{,}823$ \\
			\addlinespace
			F10k-7 & $10{,}000$ & $7$ & $124{,}800{,}000$ & $20{,}420{,}874$ & $104{,}379{,}126$ \\
			\addlinespace
			F1k-8 & $1{,}000$ & $8$ & $100{,}000{,}000$ & $28{,}910{,}930$ & $71{,}089{,}070$ \\
			\addlinespace
			F1k-9 & $1{,}000$ & $9$ & $100{,}000{,}000$ & $28{,}436{,}401$ & $71{,}563{,}599$ \\
			\addlinespace
			F1k-10 & $1{,}000$ & $10$ & $100{,}000{,}000$ & $28{,}037{,}830$ & $71{,}962{,}170$ \\
			\addlinespace
			F1k-11 & $1{,}000$ & $11$ & $100{,}000{,}000$ & $27{,}668{,}521$ & $72{,}331{,}479$ \\
			\addlinespace
			F1k-12 & $1{,}000$ & $12$ & $100{,}000{,}000$ & $27{,}351{,}931$ & $72{,}648{,}069$ \\
			\midrule
			\multicolumn{3}{r}{Total} & $2{,}135{,}580{,}000$ & $428{,}422{,}029$ & $1{,}707{,}157{,}971$ \\
			\bottomrule
		\end{tabular*}
		
		\vspace{10pt}
		\caption{Representative blocks from the  prime factor experiments with $n$ fixed and $N$ varying. The
			minimum and maximum deviations are measured relative to the uniform mean.}
		\label{tab:factor-pool-comparison}
		\scriptsize
		\resizebox{\textwidth}{!}{%
			\begin{tabular}{ccrrrrrrrr}
				\toprule
				\makecell[c]{$n$}
				& \makecell[c]{Dimensions}
				& \makecell[c]{$N$}
				& \makecell[c]{Samples}
				& \makecell[c]{Min.\ \% dev.}
				& \makecell[c]{Max.\ \% dev.}
				& \makecell[c]{$\stdev$ ratio}
				& \makecell[c]{$\MSE$}
				& \makecell[c]{Max.\ entry\\discrepancy}
				& \makecell[c]{Max.\ rank\\discrepancy} \\
				\midrule
				$2^\ast$ & $2\times1$ & $1{,}000$ & $25{,}099{,}773$
				& $-1.99\%$ & $+1.53\%$
				& $22.685$ & $2.02\mathbin{\cdot}10^{-6}$
				& $5.685\mathbin{\cdot}10^{-3}$ & $1.597\mathbin{\cdot}10^{-3}$ \\
				& & $10{,}000$ & $27{,}317{,}390$
				& $-0.40\%$ & $+0.32\%$
				& $4.395$ & $6.98\mathbin{\cdot}10^{-8}$
				& $8.991\mathbin{\cdot}10^{-4}$ & $6.833\mathbin{\cdot}10^{-5}$ \\
				& & $100{,}000$ & $28{,}635{,}104$
				& $-0.10\%$ & $+0.10\%$
				& $1.212$ & $5.07\mathbin{\cdot}10^{-9}$
				& $2.384\mathbin{\cdot}10^{-4}$ & $1.228\mathbin{\cdot}10^{-5}$ \\
				\addlinespace
				
				$3$ & $2\times2$ & $1{,}000$ & $28{,}047{,}255$
				& $-3.62\%$ & $+1.96\%$
				& $7.454$ & $2.42\mathbin{\cdot}10^{-8}$
				& $4.360\mathbin{\cdot}10^{-3}$ & $2.693\mathbin{\cdot}10^{-3}$ \\
				& & $10{,}000$ & $29{,}538{,}488$
				& $-0.64\%$ & $+0.49\%$
				& $1.548$ & $9.89\mathbin{\cdot}10^{-10}$
				& $4.857\mathbin{\cdot}10^{-4}$ & $3.689\mathbin{\cdot}10^{-4}$ \\
				& & $100{,}000$ & $30{,}339{,}126$
				& $-0.33\%$ & $+0.32\%$
				& $0.972$ & $3.80\mathbin{\cdot}10^{-10}$
				& $2.266\mathbin{\cdot}10^{-4}$ & $3.713\mathbin{\cdot}10^{-5}$ \\
				\addlinespace
				
				$3$ & $3\times2$ & $1{,}000$ & $9{,}070{,}355$
				& $-10.92\%$ & $+6.32\%$
				& $2.506$ & $9.49\mathbin{\cdot}10^{-10}$
				& $4.480\mathbin{\cdot}10^{-3}$ & $2.906\mathbin{\cdot}10^{-3}$ \\
				& & $10{,}000$ & $11{,}410{,}675$
				& $-2.65\%$ & $+3.05\%$
				& $1.080$ & $1.40\mathbin{\cdot}10^{-10}$
				& $7.249\mathbin{\cdot}10^{-4}$ & $4.430\mathbin{\cdot}10^{-4}$ \\
				& & $100{,}000$ & $13{,}156{,}873$
				& $-2.18\%$ & $+2.52\%$
				& $0.958$ & $9.55\mathbin{\cdot}10^{-11}$
				& $2.909\mathbin{\cdot}10^{-4}$ & $2.372\mathbin{\cdot}10^{-5}$ \\
				\addlinespace
				$4$ & $2\times3$ & $1{,}000$ & $29{,}630{,}608$
				& $-5.70\%$ & $+3.03\%$
				& $2.998$ & $4.16\mathbin{\cdot}10^{-10}$
				& $4.856\mathbin{\cdot}10^{-3}$ & $1.745\mathbin{\cdot}10^{-3}$ \\
				& & $10{,}000$ & $33{,}815{,}155$
				& $-1.57\%$ & $+1.88\%$
				& $1.174$ & $5.58\mathbin{\cdot}10^{-11}$
				& $7.711\mathbin{\cdot}10^{-4}$ & $3.580\mathbin{\cdot}10^{-4}$ \\
				& & $100{,}000$ & $31{,}318{,}687$
				& $-1.56\%$ & $+1.43\%$
				& $0.991$ & $4.29\mathbin{\cdot}10^{-11}$
				& $1.700\mathbin{\cdot}10^{-4}$ & $6.354\mathbin{\cdot}10^{-5}$ \\
				\addlinespace
				$5$ & $2\times4$ & $1{,}000$ & $31{,}593{,}318$
				& $-7.52\%$ & $+6.68\%$
				& $1.497$ & $1.08\mathbin{\cdot}10^{-11}$
				& $5.417\mathbin{\cdot}10^{-3}$ & $7.541\mathbin{\cdot}10^{-4}$ \\
				& & $10{,}000$ & $33{,}141{,}355$
				& $-5.87\%$ & $+5.14\%$
				& $1.040$ & $4.98\mathbin{\cdot}10^{-12}$
				& $9.651\mathbin{\cdot}10^{-4}$ & $1.426\mathbin{\cdot}10^{-4}$ \\
				& & $100{,}000$ & $33{,}668{,}895$
				& $-5.14\%$ & $+6.16\%$
				& $0.999$ & $4.52\mathbin{\cdot}10^{-12}$
				& $1.950\mathbin{\cdot}10^{-4}$ & $1.950\mathbin{\cdot}10^{-5}$ \\
				\addlinespace
				
				$6$ & $2\times5$ & $1{,}000$ & $38{,}339{,}437$
				& $-17.45\%$ & $+19.05\%$
				& $1.098$ & $5.33\mathbin{\cdot}10^{-13}$
				& $5.655\mathbin{\cdot}10^{-3}$ & $2.789\mathbin{\cdot}10^{-4}$ \\
				& & $10{,}000$ & $40{,}899{,}399$
				& $-15.83\%$ & $+15.93\%$
				& $1.004$ & $4.17\mathbin{\cdot}10^{-13}$
				& $9.322\mathbin{\cdot}10^{-4}$ & $3.463\mathbin{\cdot}10^{-5}$ \\
				& & $100{,}000$ & $42{,}714{,}501$
				& $-15.53\%$ & $+16.68\%$
				& $0.997$ & $3.94\mathbin{\cdot}10^{-13}$
				& $2.365\mathbin{\cdot}10^{-4}$ & $1.529\mathbin{\cdot}10^{-5}$ \\
				\bottomrule
		\end{tabular}}
	\end{table}
	
	The $2\times 6$ block is particularly informative, since it is the largest matrix space we considered. Every possible matrix occurs, the maximal entry and rank discrepancies are tiny, and the observed standard deviation is almost exactly equal to the uniform standard deviation. The wide relative range between the maximum and minimum counts is compatible with uniform sampling over a large matrix space of size $3^{12}$. The largest count is around $5.11$ standard deviations from the mean. In the uniform multinomial distribution, the count of any fixed matrix is approximately normal, so the expected number of matrices whose counts exceed $k$ standard deviations from the mean is on the order of $3^{12}e^{-k^2/2}$. This becomes $1$ when $\sqrt{2\log 3^{12}}\approx 5.13$, so fluctuations of the observed size are expected in a matrix space this large.
	
	The prime factor experiments also allow us to test the convergence in \Cref{conj:componentwise}. Fixing the matrix dimensions $(m+t)\times m$ is equivalent to fixing $n = \omega(B)$ and the global Selmer ratio, and \Cref{conj:componentwise} proposes that, with these variables fixed, the matrices $\mab'$ become uniform as $H(\eab)\to\infty$. We can approximate the effect of increasing the height while fixing the matrix dimensions by fixing $n$ and allowing $N$, the size of the prime pool, to vary. As $N$ grows, the size of $B$, and therefore the height of $\eab$ also typically grows, and the likelihood of any fixed prime dividing $B$ decreases. In \Cref{tab:factor-pool-comparison}, we summarise this data from several representative, well-populated blocks. For every $n$, despite the sample sizes being comparable, the matrices become progressively more uniform as $N$ grows, as the underlying heuristic of \Cref{conj:componentwise} predicts.
	
	\subsubsection{Summary}
	
	Taken together, our experiments give a coherent picture of the distribution of $\mab'$ that supports \Cref{conj:componentwise}. The sifted height experiments suggest that the discrepancies visible at computationally accessible heights are naturally explained by the disproportionate impact of small primes at these scales. A modest increase in the size of the smallest prime factor of $B(A^3-27B)$ produces a distribution that is substantially more uniform.
	The data from the prime factor experiments is considerably more uniform, and for fixed matrix dimensions, that uniformity improves when the size of the prime pool increases. The uniformity persists both in small and large matrix spaces. The data therefore provides evidence for the componentwise equidistribution of \Cref{conj:componentwise}, and hence for the random-matrix mechanism that underlies the nullity distribution predicted in \Cref{conj:intro}.

	\section*{Statement on use of generative AI}
	
	We used ChatGPT 5.5 and 5.6 to optimise and parallelise our code in SageMath, search the literature, suggest statistical strategies to analyse the data, create tables from the data, check our results, write documentation for our code, and suggest improvements to the manuscript.
	
	\section*{Acknowledgments}
	
	We are grateful to The Ohio State University's Cycle program, and in particular Kacey Aurum, for introducing us to each other and providing a framework for the early stages of this project. Computational resources for this project were provided by NSF ACCESS grant MTH250073. AW was supported by an AMS-Simons travel grant.
	
	\clearpage
	\appendix
	
\begin{landscape}
\section{Data and tables}\label{appendix}
We include the overview tables from the experiments not shown in the body. All raw data is available at ***github***

\tiny
\setlength{\tabcolsep}{2pt}
\begin{longtable}{crrrrrrrrrrrrrr}
\caption{Distribution statistics for all sifted height experiments.  The last two columns are absolute discrepancies in probability.}\label{tab:full-height-statistics}\\
\toprule
\makecell[c]{Dimensions}
& \makecell[c]{Total generated}
& \makecell[c]{Possible matrices}
& \makecell[c]{Mean count}
& \makecell[c]{Observed\\std.\ dev.}
& \makecell[c]{Uniform\\std.\ dev.}
& \makecell[c]{SD\\ratio}
& \makecell[c]{Minimum}
& \makecell[c]{\% dev.}
& \makecell[c]{Maximum}
& \makecell[c]{\% dev.}
& \makecell[c]{$\MSE$}
& \makecell[c]{Max. entry\\discrepancy}
& \makecell[c]{Max. rank\\discrepancy} \\
\midrule
\endfirsthead
\multicolumn{14}{c}{\tablename\ \thetable\ continued}\\
\toprule
\makecell[c]{Dimensions}
& \makecell[c]{Total generated}
& \makecell[c]{Possible matrices}
& \makecell[c]{Mean count}
& \makecell[c]{Observed\\std.\ dev.}
& \makecell[c]{Uniform\\std.\ dev.}
& \makecell[c]{SD\\ratio}
& \makecell[c]{Minimum}
& \makecell[c]{\% dev.}
& \makecell[c]{Maximum}
& \makecell[c]{\% dev.}
& \makecell[c]{$\MSE$}
& \makecell[c]{Max. entry\\discrepancy}
& \makecell[c]{Max. rank\\discrepancy} \\
\midrule
\endhead
\midrule
\multicolumn{14}{r}{Continued on next page}\\
\endfoot
\bottomrule
\endlastfoot
\multicolumn{14}{l}{\textbf{H-1000, $p<100$}} \\
\addlinespace[2pt]
$1\times1$
& $295{,}776{,}404$
& $3$
& $98{,}592{,}134.67$
& $347{,}673.97$
& $8{,}107.29$
& $42.884$
& $98{,}100{,}944$
& $-0.50\%$
& $98{,}856{,}824$
& $+0.27\%$
& $1.38\mathbin{\cdot}10^{-6}$
& $1.661\mathbin{\cdot}10^{-3}$
& $1.661\mathbin{\cdot}10^{-3}$ \\

$2\times1$
& $110{,}008{,}640$
& $9$
& $12{,}223{,}182.22$
& $134{,}890.37$
& $3{,}296.22$
& $40.923$
& $12{,}010{,}717$
& $-1.74\%$
& $12{,}338{,}125$
& $+0.94\%$
& $1.50\mathbin{\cdot}10^{-6}$
& $5.150\mathbin{\cdot}10^{-3}$
& $1.931\mathbin{\cdot}10^{-3}$ \\

$3\times1$
& $11{,}816{,}765$
& $27$
& $437{,}657.96$
& $8{,}811.23$
& $649.19$
& $13.573$
& $418{,}454$
& $-4.39\%$
& $448{,}605$
& $+2.50\%$
& $5.56\mathbin{\cdot}10^{-7}$
& $8.065\mathbin{\cdot}10^{-3}$
& $1.625\mathbin{\cdot}10^{-3}$ \\

$4\times1$
& $212{,}101$
& $81$
& $2{,}618.53$
& $104.93$
& $50.85$
& $2.063$
& $2{,}301$
& $-12.13\%$
& $2{,}818$
& $+7.62\%$
& $2.45\mathbin{\cdot}10^{-7}$
& $1.002\mathbin{\cdot}10^{-2}$
& $1.497\mathbin{\cdot}10^{-3}$ \\

$1\times2$
& $58{,}592{,}850$
& $9$
& $6{,}510{,}316.67$
& $58{,}499.16$
& $2{,}405.61$
& $24.318$
& $6{,}393{,}505$
& $-1.79\%$
& $6{,}576{,}129$
& $+1.01\%$
& $9.97\mathbin{\cdot}10^{-7}$
& $3.594\mathbin{\cdot}10^{-3}$
& $1.994\mathbin{\cdot}10^{-3}$ \\

$2\times2$
& $21{,}790{,}143$
& $81$
& $269{,}014.11$
& $7{,}731.24$
& $515.45$
& $14.999$
& $250{,}466$
& $-6.89\%$
& $278{,}995$
& $+3.71\%$
& $1.26\mathbin{\cdot}10^{-7}$
& $1.128\mathbin{\cdot}10^{-2}$
& $4.870\mathbin{\cdot}10^{-3}$ \\

$3\times2$
& $2{,}343{,}106$
& $729$
& $3{,}214.14$
& $182.07$
& $56.65$
& $3.214$
& $2{,}624$
& $-18.36\%$
& $3{,}588$
& $+11.63\%$
& $6.04\mathbin{\cdot}10^{-9}$
& $1.892\mathbin{\cdot}10^{-2}$
& $5.118\mathbin{\cdot}10^{-3}$ \\

$4\times2$
& $41{,}934$
& $6{,}561$
& $6.39$
& $2.65$
& $2.53$
& $1.050$
& $0$
& $-100.00\%$
& $20$
& $+212.92\%$
& $4.01\mathbin{\cdot}10^{-9}$
& $2.595\mathbin{\cdot}10^{-2}$
& $3.330\mathbin{\cdot}10^{-3}$ \\

$1\times3$
& $417{,}551$
& $27$
& $15{,}464.85$
& $341.47$
& $122.03$
& $2.798$
& $14{,}686$
& $-5.04\%$
& $16{,}239$
& $+5.01\%$
& $6.69\mathbin{\cdot}10^{-7}$
& $7.268\mathbin{\cdot}10^{-3}$
& $1.865\mathbin{\cdot}10^{-3}$ \\

$2\times3$
& $155{,}833$
& $729$
& $213.76$
& $19.48$
& $14.61$
& $1.333$
& $143$
& $-33.10\%$
& $288$
& $+34.73\%$
& $1.56\mathbin{\cdot}10^{-8}$
& $1.730\mathbin{\cdot}10^{-2}$
& $4.897\mathbin{\cdot}10^{-3}$ \\

$3\times3$
& $16{,}631$
& $19{,}683$
& $0.845$
& $0.926$
& $0.919$
& $1.007$
& $0$
& $-100.00\%$
& $8$
& $+846.81\%$
& $3.10\mathbin{\cdot}10^{-9}$
& $3.028\mathbin{\cdot}10^{-2}$
& $1.381\mathbin{\cdot}10^{-2}$ \\

$4\times3$
& $300$
& $531{,}441$
& $0.00056$
& $0.0238$
& $0.0238$
& $1.000$
& $0$
& $-100.00\%$
& $1$
& $+177{,}047.00\%$
& $6.27\mathbin{\cdot}10^{-9}$
& $8.667\mathbin{\cdot}10^{-2}$
& $2.127\mathbin{\cdot}10^{-2}$ \\

\addlinespace
\multicolumn{14}{l}{\textbf{H-1000, $p<300$}} \\
\addlinespace[2pt]
$1\times1$
& $187{,}537{,}358$
& $3$
& $62{,}512{,}452.67$
& $93{,}819.23$
& $6{,}455.62$
& $14.533$
& $62{,}380{,}467$
& $-0.21\%$
& $62{,}590{,}189$
& $+0.12\%$
& $2.50\mathbin{\cdot}10^{-7}$
& $7.038\mathbin{\cdot}10^{-4}$
& $7.038\mathbin{\cdot}10^{-4}$ \\

$2\times1$
& $57{,}663{,}276$
& $9$
& $6{,}407{,}030.67$
& $33{,}903.96$
& $2{,}386.45$
& $14.207$
& $6{,}356{,}434$
& $-0.79\%$
& $6{,}434{,}415$
& $+0.43\%$
& $3.46\mathbin{\cdot}10^{-7}$
& $2.488\mathbin{\cdot}10^{-3}$
& $7.918\mathbin{\cdot}10^{-4}$ \\

$3\times1$
& $3{,}605{,}233$
& $27$
& $133{,}527.15$
& $1{,}572.76$
& $358.58$
& $4.386$
& $130{,}251$
& $-2.45\%$
& $135{,}781$
& $+1.69\%$
& $1.90\mathbin{\cdot}10^{-7}$
& $3.806\mathbin{\cdot}10^{-3}$
& $7.603\mathbin{\cdot}10^{-4}$ \\

$4\times1$
& $1{,}390$
& $81$
& $17.16$
& $4.49$
& $4.12$
& $1.091$
& $7$
& $-59.21\%$
& $32$
& $+86.47\%$
& $1.04\mathbin{\cdot}10^{-5}$
& $2.638\mathbin{\cdot}10^{-2}$
& $6.040\mathbin{\cdot}10^{-4}$ \\

$1\times2$
& $11{,}532{,}687$
& $9$
& $1{,}281{,}409.67$
& $5{,}928.15$
& $1{,}067.25$
& $5.555$
& $1{,}270{,}220$
& $-0.87\%$
& $1{,}289{,}020$
& $+0.59\%$
& $2.64\mathbin{\cdot}10^{-7}$
& $1.812\mathbin{\cdot}10^{-3}$
& $9.703\mathbin{\cdot}10^{-4}$ \\

$2\times2$
& $3{,}560{,}944$
& $81$
& $43{,}962.27$
& $863.60$
& $208.37$
& $4.144$
& $41{,}624$
& $-5.32\%$
& $45{,}275$
& $+2.99\%$
& $5.88\mathbin{\cdot}10^{-8}$
& $7.390\mathbin{\cdot}10^{-3}$
& $3.448\mathbin{\cdot}10^{-3}$ \\

$3\times2$
& $226{,}846$
& $729$
& $311.17$
& $21.44$
& $17.63$
& $1.216$
& $252$
& $-19.02\%$
& $386$
& $+24.05\%$
& $8.94\mathbin{\cdot}10^{-9}$
& $1.444\mathbin{\cdot}10^{-2}$
& $3.052\mathbin{\cdot}10^{-3}$ \\

$4\times2$
& $83$
& $6{,}561$
& $0.0127$
& $0.113$
& $0.112$
& $1.006$
& $0$
& $-100.00\%$
& $2$
& $+15{,}709.64\%$
& $1.86\mathbin{\cdot}10^{-6}$
& $1.004\mathbin{\cdot}10^{-1}$
& $2.352\mathbin{\cdot}10^{-2}$ \\

\addlinespace
\multicolumn{14}{l}{\textbf{H-1000, $p<500$}} \\
\addlinespace[2pt]
$1\times1$
& $151{,}443{,}534$
& $3$
& $50{,}481{,}178.00$
& $41{,}689.66$
& $5{,}801.22$
& $7.186$
& $50{,}422{,}249$
& $-0.12\%$
& $50{,}512{,}246$
& $+0.06\%$
& $7.58\mathbin{\cdot}10^{-8}$
& $3.891\mathbin{\cdot}10^{-4}$
& $3.891\mathbin{\cdot}10^{-4}$ \\

$2\times1$
& $41{,}999{,}326$
& $9$
& $4{,}666{,}591.78$
& $14{,}454.52$
& $2{,}036.68$
& $7.097$
& $4{,}643{,}489$
& $-0.50\%$
& $4{,}678{,}561$
& $+0.26\%$
& $1.18\mathbin{\cdot}10^{-7}$
& $1.447\mathbin{\cdot}10^{-3}$
& $3.899\mathbin{\cdot}10^{-4}$ \\

$3\times1$
& $1{,}765{,}840$
& $27$
& $65{,}401.48$
& $595.44$
& $250.96$
& $2.373$
& $64{,}350$
& $-1.61\%$
& $66{,}409$
& $+1.54\%$
& $1.14\mathbin{\cdot}10^{-7}$
& $2.910\mathbin{\cdot}10^{-3}$
& $3.106\mathbin{\cdot}10^{-4}$ \\

$1\times2$
& $3{,}231{,}567$
& $9$
& $359{,}063.00$
& $1{,}102.98$
& $564.95$
& $1.952$
& $357{,}190$
& $-0.52\%$
& $360{,}821$
& $+0.49\%$
& $1.16\mathbin{\cdot}10^{-7}$
& $1.065\mathbin{\cdot}10^{-3}$
& $5.796\mathbin{\cdot}10^{-4}$ \\

$2\times2$
& $901{,}892$
& $81$
& $11{,}134.47$
& $173.44$
& $104.87$
& $1.654$
& $10{,}663$
& $-4.23\%$
& $11{,}525$
& $+3.51\%$
& $3.70\mathbin{\cdot}10^{-8}$
& $4.312\mathbin{\cdot}10^{-3}$
& $2.631\mathbin{\cdot}10^{-3}$ \\

$3\times2$
& $39{,}226$
& $729$
& $53.81$
& $7.57$
& $7.33$
& $1.033$
& $35$
& $-34.95\%$
& $83$
& $+54.25\%$
& $3.73\mathbin{\cdot}10^{-8}$
& $1.157\mathbin{\cdot}10^{-2}$
& $3.080\mathbin{\cdot}10^{-3}$ \\

\addlinespace
\multicolumn{14}{l}{\textbf{H-1000, $p<800$}} \\
\addlinespace[2pt]
$1\times1$
& $124{,}309{,}856$
& $3$
& $41{,}436{,}618.67$
& $32{,}902.34$
& $5{,}255.89$
& $6.260$
& $41{,}391{,}323$
& $-0.11\%$
& $41{,}468{,}490$
& $+0.08\%$
& $7.01\mathbin{\cdot}10^{-8}$
& $3.644\mathbin{\cdot}10^{-4}$
& $3.644\mathbin{\cdot}10^{-4}$ \\

$2\times1$
& $31{,}604{,}006$
& $9$
& $3{,}511{,}556.22$
& $11{,}906.36$
& $1{,}766.74$
& $6.739$
& $3{,}493{,}872$
& $-0.50\%$
& $3{,}526{,}017$
& $+0.41\%$
& $1.42\mathbin{\cdot}10^{-7}$
& $1.512\mathbin{\cdot}10^{-3}$
& $3.935\mathbin{\cdot}10^{-4}$ \\

$3\times1$
& $825{,}049$
& $27$
& $30{,}557.37$
& $326.69$
& $171.54$
& $1.904$
& $29{,}903$
& $-2.14\%$
& $31{,}190$
& $+2.07\%$
& $1.57\mathbin{\cdot}10^{-7}$
& $2.678\mathbin{\cdot}10^{-3}$
& $3.107\mathbin{\cdot}10^{-4}$ \\

$1\times2$
& $306{,}603$
& $9$
& $34{,}067.00$
& $244.01$
& $174.02$
& $1.402$
& $33{,}740$
& $-0.96\%$
& $34{,}529$
& $+1.36\%$
& $6.33\mathbin{\cdot}10^{-7}$
& $2.635\mathbin{\cdot}10^{-3}$
& $7.371\mathbin{\cdot}10^{-4}$ \\

$2\times2$
& $78{,}666$
& $81$
& $971.19$
& $34.48$
& $30.97$
& $1.113$
& $885$
& $-8.87\%$
& $1{,}062$
& $+9.35\%$
& $1.92\mathbin{\cdot}10^{-7}$
& $7.589\mathbin{\cdot}10^{-3}$
& $4.565\mathbin{\cdot}10^{-3}$ \\

$3\times2$
& $2{,}098$
& $729$
& $2.88$
& $1.71$
& $1.70$
& $1.010$
& $0$
& $-100.00\%$
& $10$
& $+247.47\%$
& $6.66\mathbin{\cdot}10^{-7}$
& $1.779\mathbin{\cdot}10^{-2}$
& $7.713\mathbin{\cdot}10^{-3}$ \\

\addlinespace
\multicolumn{14}{l}{\textbf{H-1200, $p<500$}} \\
\addlinespace[2pt]
$1\times1$
& $2{,}836{,}897{,}251$
& $3$
& $945{,}632{,}417.00$
& $702{,}321.60$
& $25{,}108.20$
& $27.972$
& $944{,}640{,}172$
& $-0.10\%$
& $946{,}166{,}891$
& $+0.06\%$
& $6.13\mathbin{\cdot}10^{-8}$
& $3.498\mathbin{\cdot}10^{-4}$
& $3.498\mathbin{\cdot}10^{-4}$ \\

$2\times1$
& $815{,}027{,}639$
& $9$
& $90{,}558{,}626.56$
& $239{,}007.16$
& $8{,}971.99$
& $26.639$
& $90{,}212{,}049$
& $-0.38\%$
& $90{,}749{,}739$
& $+0.21\%$
& $8.60\mathbin{\cdot}10^{-8}$
& $1.242\mathbin{\cdot}10^{-3}$
& $4.252\mathbin{\cdot}10^{-4}$ \\

$3\times1$
& $39{,}739{,}180$
& $27$
& $1{,}471{,}821.48$
& $9{,}565.32$
& $1{,}190.51$
& $8.035$
& $1{,}453{,}826$
& $-1.22\%$
& $1{,}486{,}299$
& $+0.98\%$
& $5.79\mathbin{\cdot}10^{-8}$
& $2.208\mathbin{\cdot}10^{-3}$
& $1.894\mathbin{\cdot}10^{-4}$ \\

$1\times2$
& $99{,}775{,}184$
& $9$
& $11{,}086{,}131.56$
& $32{,}730.33$
& $3{,}139.16$
& $10.426$
& $11{,}027{,}245$
& $-0.53\%$
& $11{,}134{,}215$
& $+0.43\%$
& $1.08\mathbin{\cdot}10^{-7}$
& $1.128\mathbin{\cdot}10^{-3}$
& $5.902\mathbin{\cdot}10^{-4}$ \\

$2\times2$
& $28{,}884{,}156$
& $81$
& $356{,}594.52$
& $3{,}565.56$
& $593.46$
& $6.008$
& $347{,}083$
& $-2.67\%$
& $362{,}712$
& $+1.72\%$
& $1.52\mathbin{\cdot}10^{-8}$
& $3.237\mathbin{\cdot}10^{-3}$
& $1.417\mathbin{\cdot}10^{-3}$ \\

$3\times2$
& $1{,}445{,}809$
& $729$
& $1{,}983.28$
& $67.72$
& $44.50$
& $1.522$
& $1{,}781$
& $-10.20\%$
& $2{,}176$
& $+9.72\%$
& $2.19\mathbin{\cdot}10^{-9}$
& $8.403\mathbin{\cdot}10^{-3}$
& $1.055\mathbin{\cdot}10^{-3}$ \\

\addlinespace
\multicolumn{14}{l}{\textbf{H-2250, $p<2000$}} \\
\addlinespace[2pt]
$1\times1$
& $1{,}070{,}315{,}260$
& $3$
& $356{,}771{,}753.33$
& $119{,}206.57$
& $15{,}422.32$
& $7.729$
& $356{,}603{,}184$
& $-0.05\%$
& $356{,}857{,}934$
& $+0.02\%$
& $1.24\mathbin{\cdot}10^{-8}$
& $1.575\mathbin{\cdot}10^{-4}$
& $1.575\mathbin{\cdot}10^{-4}$ \\

$2\times1$
& $262{,}760{,}745$
& $9$
& $29{,}195{,}638.33$
& $39{,}730.94$
& $5{,}094.28$
& $7.799$
& $29{,}135{,}075$
& $-0.21\%$
& $29{,}244{,}180$
& $+0.17\%$
& $2.29\mathbin{\cdot}10^{-8}$
& $6.188\mathbin{\cdot}10^{-4}$
& $1.776\mathbin{\cdot}10^{-4}$ \\

$3\times1$
& $5{,}158{,}956$
& $27$
& $191{,}072.44$
& $861.13$
& $428.95$
& $2.008$
& $189{,}509$
& $-0.82\%$
& $193{,}034$
& $+1.03\%$
& $2.79\mathbin{\cdot}10^{-8}$
& $9.453\mathbin{\cdot}10^{-4}$
& $8.346\mathbin{\cdot}10^{-5}$ \\

$1\times2$
& $590{,}099$
& $9$
& $65{,}566.56$
& $199.26$
& $241.42$
& $0.825$
& $65{,}312$
& $-0.39\%$
& $65{,}911$
& $+0.53\%$
& $1.14\mathbin{\cdot}10^{-7}$
& $4.259\mathbin{\cdot}10^{-4}$
& $2.823\mathbin{\cdot}10^{-4}$ \\

$2\times2$
& $146{,}108$
& $81$
& $1{,}803.80$
& $43.85$
& $42.21$
& $1.039$
& $1{,}684$
& $-6.64\%$
& $1{,}900$
& $+5.33\%$
& $9.01\mathbin{\cdot}10^{-8}$
& $3.345\mathbin{\cdot}10^{-3}$
& $5.727\mathbin{\cdot}10^{-4}$ \\

$3\times2$
& $3{,}085$
& $729$
& $4.23$
& $2.13$
& $2.06$
& $1.037$
& $0$
& $-100.00\%$
& $12$
& $+183.57\%$
& $4.77\mathbin{\cdot}10^{-7}$
& $1.610\mathbin{\cdot}10^{-2}$
& $5.945\mathbin{\cdot}10^{-3}$ \\

\addlinespace
\multicolumn{14}{l}{\textbf{H-3250, $p<1000$}} \\
\addlinespace[2pt]
$1\times1$
& $4{,}427{,}036{,}888$
& $3$
& $1{,}475{,}678{,}962.67$
& $575{,}557.64$
& $31{,}365.36$
& $18.350$
& $1{,}474{,}865{,}895$
& $-0.06\%$
& $1{,}476{,}118{,}521$
& $+0.03\%$
& $1.69\mathbin{\cdot}10^{-8}$
& $1.837\mathbin{\cdot}10^{-4}$
& $1.837\mathbin{\cdot}10^{-4}$ \\

$2\times1$
& $1{,}292{,}105{,}782$
& $9$
& $143{,}567{,}309.11$
& $194{,}822.13$
& $11{,}296.70$
& $17.246$
& $143{,}283{,}441$
& $-0.20\%$
& $143{,}716{,}189$
& $+0.10\%$
& $2.27\mathbin{\cdot}10^{-8}$
& $6.386\mathbin{\cdot}10^{-4}$
& $2.165\mathbin{\cdot}10^{-4}$ \\

$3\times1$
& $66{,}615{,}066$
& $27$
& $2{,}467{,}224.67$
& $8{,}144.83$
& $1{,}541.38$
& $5.284$
& $2{,}450{,}982$
& $-0.66\%$
& $2{,}480{,}888$
& $+0.55\%$
& $1.49\mathbin{\cdot}10^{-8}$
& $1.120\mathbin{\cdot}10^{-3}$
& $1.510\mathbin{\cdot}10^{-4}$ \\

$1\times2$
& $203{,}960{,}440$
& $9$
& $22{,}662{,}271.11$
& $26{,}159.53$
& $4{,}488.23$
& $5.828$
& $22{,}615{,}982$
& $-0.20\%$
& $22{,}704{,}317$
& $+0.19\%$
& $1.65\mathbin{\cdot}10^{-8}$
& $3.615\mathbin{\cdot}10^{-4}$
& $2.270\mathbin{\cdot}10^{-4}$ \\

$2\times2$
& $59{,}973{,}424$
& $81$
& $740{,}412.64$
& $3{,}300.93$
& $855.14$
& $3.860$
& $731{,}360$
& $-1.22\%$
& $746{,}325$
& $+0.80\%$
& $3.03\mathbin{\cdot}10^{-9}$
& $1.295\mathbin{\cdot}10^{-3}$
& $5.845\mathbin{\cdot}10^{-4}$ \\

$3\times2$
& $3{,}173{,}517$
& $729$
& $4{,}353.25$
& $82.17$
& $65.93$
& $1.246$
& $4{,}074$
& $-6.41\%$
& $4{,}597$
& $+5.60\%$
& $6.70\mathbin{\cdot}10^{-10}$
& $3.062\mathbin{\cdot}10^{-3}$
& $5.860\mathbin{\cdot}10^{-4}$ \\

\end{longtable}
\end{landscape}


\begin{landscape}
\tiny
\setlength{\tabcolsep}{2pt}
\begin{longtable}{crrrrrrrrrrrrr}
\caption{Distribution statistics for all prime factor experiments.}\\
\toprule
\makecell[c]{Dimensions}
& \makecell[c]{Total generated}
& \makecell[c]{Possible matrices}
& \makecell[c]{Mean count}
& \makecell[c]{Observed\\std.\ dev.}
& \makecell[c]{Uniform\\std.\ dev.}
& \makecell[c]{SD\\ratio}
& \makecell[c]{Minimum}
& \makecell[c]{\% dev.}
& \makecell[c]{Maximum}
& \makecell[c]{\% dev.}
& \makecell[c]{$\MSE$}
& \makecell[c]{Max. entry\\discrepancy}
& \makecell[c]{Max. rank\\discrepancy} \\
\midrule
\endfirsthead
\multicolumn{14}{c}{\tablename\ \thetable\ continued}\\
\toprule
\makecell[c]{Dimensions}
& \makecell[c]{Total generated}
& \makecell[c]{Possible matrices}
& \makecell[c]{Mean count}
& \makecell[c]{Observed\\std.\ dev.}
& \makecell[c]{Uniform\\std.\ dev.}
& \makecell[c]{SD\\ratio}
& \makecell[c]{Minimum}
& \makecell[c]{\% dev.}
& \makecell[c]{Maximum}
& \makecell[c]{\% dev.}
& \makecell[c]{$\MSE$}
& \makecell[c]{Max. entry\\discrepancy}
& \makecell[c]{Max. rank\\discrepancy} \\
\midrule
\endhead
\midrule
\multicolumn{14}{r}{Continued on next page}\\
\endfoot
\bottomrule
\endlastfoot
\multicolumn{14}{l}{\textbf{F1k-2$^\ast$}} \\
\addlinespace[2pt]
$1\times1$
& $44{,}111{,}474$
& $3$
& $14{,}703{,}824.67$
& $59{,}419.39$
& $3{,}130.90$
& $18.978$
& $14{,}625{,}502$
& $-0.53\%$
& $14{,}769{,}352$
& $+0.45\%$
& $1.81\mathbin{\cdot}10^{-6}$
& $1.776\mathbin{\cdot}10^{-3}$
& $1.776\mathbin{\cdot}10^{-3}$ \\

$2\times1$
& $25{,}099{,}773$
& $9$
& $2{,}788{,}863.67$
& $35{,}717.36$
& $1{,}574.48$
& $22.685$
& $2{,}733{,}287$
& $-1.99\%$
& $2{,}831{,}413$
& $+1.53\%$
& $2.02\mathbin{\cdot}10^{-6}$
& $5.685\mathbin{\cdot}10^{-3}$
& $1.597\mathbin{\cdot}10^{-3}$ \\

$3\times1$
& $5{,}872{,}394$
& $27$
& $217{,}496.07$
& $5{,}102.51$
& $457.65$
& $11.149$
& $203{,}001$
& $-6.66\%$
& $224{,}324$
& $+3.14\%$
& $7.55\mathbin{\cdot}10^{-7}$
& $5.438\mathbin{\cdot}10^{-3}$
& $8.435\mathbin{\cdot}10^{-4}$ \\

$4\times1$
& $550{,}537$
& $81$
& $6{,}796.75$
& $285.00$
& $81.93$
& $3.478$
& $5{,}763$
& $-15.21\%$
& $7{,}381$
& $+8.60\%$
& $2.68\mathbin{\cdot}10^{-7}$
& $6.489\mathbin{\cdot}10^{-3}$
& $5.517\mathbin{\cdot}10^{-4}$ \\

$5\times1$
& $16{,}056$
& $243$
& $66.07$
& $12.08$
& $8.11$
& $1.489$
& $40$
& $-39.46\%$
& $98$
& $+48.32\%$
& $5.66\mathbin{\cdot}10^{-7}$
& $1.158\mathbin{\cdot}10^{-2}$
& $1.126\mathbin{\cdot}10^{-3}$ \\

$6\times1$
& $48$
& $729$
& $0.0658$
& $0.303$
& $0.256$
& $1.181$
& $0$
& $-100.00\%$
& $3$
& $+4{,}456.25\%$
& $3.98\mathbin{\cdot}10^{-5}$
& $1.458\mathbin{\cdot}10^{-1}$
& $1.372\mathbin{\cdot}10^{-3}$ \\

\addlinespace
\multicolumn{14}{l}{\textbf{F10k-2$^\ast$}} \\
\addlinespace[2pt]
$1\times1$
& $41{,}809{,}445$
& $3$
& $13{,}936{,}481.67$
& $6{,}135.77$
& $3{,}048.11$
& $2.013$
& $13{,}929{,}503$
& $-0.05\%$
& $13{,}944{,}437$
& $+0.06\%$
& $2.15\mathbin{\cdot}10^{-8}$
& $1.903\mathbin{\cdot}10^{-4}$
& $1.669\mathbin{\cdot}10^{-4}$ \\

$2\times1$
& $27{,}317{,}390$
& $9$
& $3{,}035{,}265.56$
& $7{,}219.15$
& $1{,}642.56$
& $4.395$
& $3{,}023{,}159$
& $-0.40\%$
& $3{,}045{,}020$
& $+0.32\%$
& $6.98\mathbin{\cdot}10^{-8}$
& $8.991\mathbin{\cdot}10^{-4}$
& $6.833\mathbin{\cdot}10^{-5}$ \\

$3\times1$
& $8{,}114{,}492$
& $27$
& $300{,}536.74$
& $1{,}387.16$
& $537.96$
& $2.579$
& $296{,}504$
& $-1.34\%$
& $303{,}109$
& $+0.86\%$
& $2.92\mathbin{\cdot}10^{-8}$
& $1.259\mathbin{\cdot}10^{-3}$
& $4.643\mathbin{\cdot}10^{-5}$ \\

$4\times1$
& $1{,}160{,}774$
& $81$
& $14{,}330.54$
& $182.89$
& $118.97$
& $1.537$
& $13{,}692$
& $-4.46\%$
& $14{,}705$
& $+2.61\%$
& $2.48\mathbin{\cdot}10^{-8}$
& $2.750\mathbin{\cdot}10^{-3}$
& $2.089\mathbin{\cdot}10^{-4}$ \\

$5\times1$
& $74{,}144$
& $243$
& $305.12$
& $17.98$
& $17.43$
& $1.031$
& $258$
& $-15.44\%$
& $353$
& $+15.69\%$
& $5.88\mathbin{\cdot}10^{-8}$
& $3.961\mathbin{\cdot}10^{-3}$
& $5.556\mathbin{\cdot}10^{-5}$ \\

$6\times1$
& $1{,}623$
& $729$
& $2.23$
& $1.51$
& $1.49$
& $1.015$
& $0$
& $-100.00\%$
& $8$
& $+259.33\%$
& $8.69\mathbin{\cdot}10^{-7}$
& $2.773\mathbin{\cdot}10^{-2}$
& $1.395\mathbin{\cdot}10^{-4}$ \\

$7\times1$
& $11$
& $2{,}187$
& $0.00503$
& $0.0707$
& $0.0709$
& $0.998$
& $0$
& $-100.00\%$
& $1$
& $+19{,}781.82\%$
& $4.14\mathbin{\cdot}10^{-5}$
& $3.939\mathbin{\cdot}10^{-1}$
& $4.572\mathbin{\cdot}10^{-4}$ \\

\addlinespace
\multicolumn{14}{l}{\textbf{F100k-2$^\ast$}} \\
\addlinespace[2pt]
$1\times1$
& $39{,}933{,}999$
& $3$
& $13{,}311{,}333.00$
& $958.07$
& $2{,}978.96$
& $0.322$
& $13{,}309{,}979$
& $-0.01\%$
& $13{,}312{,}053$
& $+0.01\%$
& $5.76\mathbin{\cdot}10^{-10}$
& $3.391\mathbin{\cdot}10^{-5}$
& $1.803\mathbin{\cdot}10^{-5}$ \\

$2\times1$
& $28{,}635{,}104$
& $9$
& $3{,}181{,}678.22$
& $2{,}039.06$
& $1{,}681.71$
& $1.212$
& $3{,}178{,}406$
& $-0.10\%$
& $3{,}184{,}812$
& $+0.10\%$
& $5.07\mathbin{\cdot}10^{-9}$
& $2.384\mathbin{\cdot}10^{-4}$
& $1.228\mathbin{\cdot}10^{-5}$ \\

$3\times1$
& $9{,}896{,}073$
& $27$
& $366{,}521.22$
& $726.31$
& $594.09$
& $1.223$
& $364{,}975$
& $-0.42\%$
& $368{,}374$
& $+0.51\%$
& $5.39\mathbin{\cdot}10^{-9}$
& $3.988\mathbin{\cdot}10^{-4}$
& $2.043\mathbin{\cdot}10^{-6}$ \\

$4\times1$
& $1{,}798{,}197$
& $81$
& $22{,}199.96$
& $152.26$
& $148.07$
& $1.028$
& $21{,}877$
& $-1.45\%$
& $22{,}644$
& $+2.00\%$
& $7.17\mathbin{\cdot}10^{-9}$
& $8.859\mathbin{\cdot}10^{-4}$
& $8.399\mathbin{\cdot}10^{-5}$ \\

$5\times1$
& $170{,}503$
& $243$
& $701.66$
& $29.02$
& $26.43$
& $1.098$
& $632$
& $-9.93\%$
& $776$
& $+10.60\%$
& $2.90\mathbin{\cdot}10^{-8}$
& $2.162\mathbin{\cdot}10^{-3}$
& $3.719\mathbin{\cdot}10^{-5}$ \\

$6\times1$
& $7{,}403$
& $729$
& $10.16$
& $3.11$
& $3.18$
& $0.978$
& $3$
& $-70.46\%$
& $21$
& $+106.79\%$
& $1.77\mathbin{\cdot}10^{-7}$
& $8.735\mathbin{\cdot}10^{-3}$
& $1.141\mathbin{\cdot}10^{-4}$ \\

$7\times1$
& $158$
& $2{,}187$
& $0.0722$
& $0.266$
& $0.269$
& $0.989$
& $0$
& $-100.00\%$
& $2$
& $+2{,}668.35\%$
& $2.83\mathbin{\cdot}10^{-6}$
& $9.705\mathbin{\cdot}10^{-2}$
& $4.572\mathbin{\cdot}10^{-4}$ \\

\addlinespace
\multicolumn{14}{l}{\textbf{F1k-3}} \\
\addlinespace[2pt]
$1\times2$
& $40{,}815{,}496$
& $9$
& $4{,}535{,}055.11$
& $26{,}065.42$
& $2{,}007.77$
& $12.982$
& $4{,}479{,}769$
& $-1.22\%$
& $4{,}568{,}214$
& $+0.73\%$
& $4.08\mathbin{\cdot}10^{-7}$
& $2.236\mathbin{\cdot}10^{-3}$
& $1.355\mathbin{\cdot}10^{-3}$ \\

$2\times2$
& $28{,}047{,}255$
& $81$
& $346{,}262.41$
& $4{,}358.81$
& $584.80$
& $7.454$
& $333{,}738$
& $-3.62\%$
& $353{,}048$
& $+1.96\%$
& $2.42\mathbin{\cdot}10^{-8}$
& $4.360\mathbin{\cdot}10^{-3}$
& $2.693\mathbin{\cdot}10^{-3}$ \\

$3\times2$
& $9{,}070{,}355$
& $729$
& $12{,}442.19$
& $279.39$
& $111.47$
& $2.506$
& $11{,}084$
& $-10.92\%$
& $13{,}228$
& $+6.32\%$
& $9.49\mathbin{\cdot}10^{-10}$
& $4.480\mathbin{\cdot}10^{-3}$
& $2.906\mathbin{\cdot}10^{-3}$ \\

$4\times2$
& $1{,}488{,}229$
& $6{,}561$
& $226.83$
& $16.95$
& $15.06$
& $1.125$
& $167$
& $-26.38\%$
& $296$
& $+30.49\%$
& $1.30\mathbin{\cdot}10^{-10}$
& $4.869\mathbin{\cdot}10^{-3}$
& $2.071\mathbin{\cdot}10^{-3}$ \\

$5\times2$
& $121{,}499$
& $59{,}049$
& $2.06$
& $1.44$
& $1.43$
& $1.006$
& $0$
& $-100.00\%$
& $11$
& $+434.60\%$
& $1.41\mathbin{\cdot}10^{-10}$
& $6.277\mathbin{\cdot}10^{-3}$
& $1.612\mathbin{\cdot}10^{-3}$ \\

$6\times2$
& $4{,}263$
& $531{,}441$
& $0.00802$
& $0.0897$
& $0.0896$
& $1.001$
& $0$
& $-100.00\%$
& $2$
& $+24{,}832.72\%$
& $4.43\mathbin{\cdot}10^{-10}$
& $1.689\mathbin{\cdot}10^{-2}$
& $1.024\mathbin{\cdot}10^{-3}$ \\

\addlinespace
\multicolumn{14}{l}{\textbf{F10k-3}} \\
\addlinespace[2pt]
$1\times2$
& $38{,}286{,}564$
& $9$
& $4{,}254{,}062.67$
& $4{,}569.97$
& $1{,}944.58$
& $2.350$
& $4{,}244{,}825$
& $-0.22\%$
& $4{,}259{,}553$
& $+0.13\%$
& $1.42\mathbin{\cdot}10^{-8}$
& $4.157\mathbin{\cdot}10^{-4}$
& $2.413\mathbin{\cdot}10^{-4}$ \\

$2\times2$
& $29{,}538{,}488$
& $81$
& $364{,}672.69$
& $929.09$
& $600.14$
& $1.548$
& $362{,}354$
& $-0.64\%$
& $366{,}449$
& $+0.49\%$
& $9.89\mathbin{\cdot}10^{-10}$
& $4.857\mathbin{\cdot}10^{-4}$
& $3.689\mathbin{\cdot}10^{-4}$ \\

$3\times2$
& $11{,}410{,}675$
& $729$
& $15{,}652.50$
& $135.04$
& $125.02$
& $1.080$
& $15{,}237$
& $-2.65\%$
& $16{,}130$
& $+3.05\%$
& $1.40\mathbin{\cdot}10^{-10}$
& $7.249\mathbin{\cdot}10^{-4}$
& $4.430\mathbin{\cdot}10^{-4}$ \\

$4\times2$
& $2{,}458{,}424$
& $6{,}561$
& $374.70$
& $19.60$
& $19.36$
& $1.012$
& $309$
& $-17.53\%$
& $444$
& $+18.49\%$
& $6.35\mathbin{\cdot}10^{-11}$
& $8.272\mathbin{\cdot}10^{-4}$
& $3.692\mathbin{\cdot}10^{-4}$ \\

$5\times2$
& $302{,}508$
& $59{,}049$
& $5.12$
& $2.27$
& $2.26$
& $1.002$
& $0$
& $-100.00\%$
& $17$
& $+231.84\%$
& $5.63\mathbin{\cdot}10^{-11}$
& $1.927\mathbin{\cdot}10^{-3}$
& $2.045\mathbin{\cdot}10^{-5}$ \\

$6\times2$
& $20{,}446$
& $531{,}441$
& $0.0385$
& $0.196$
& $0.196$
& $1.000$
& $0$
& $-100.00\%$
& $4$
& $+10{,}296.97\%$
& $9.20\mathbin{\cdot}10^{-11}$
& $7.467\mathbin{\cdot}10^{-3}$
& $4.437\mathbin{\cdot}10^{-4}$ \\

\addlinespace
\multicolumn{14}{l}{\textbf{F100k-3}} \\
\addlinespace[2pt]
$1\times2$
& $36{,}317{,}337$
& $9$
& $4{,}035{,}259.67$
& $1{,}648.94$
& $1{,}893.91$
& $0.871$
& $4{,}032{,}507$
& $-0.07\%$
& $4{,}037{,}956$
& $+0.07\%$
& $2.06\mathbin{\cdot}10^{-9}$
& $5.350\mathbin{\cdot}10^{-5}$
& $7.424\mathbin{\cdot}10^{-5}$ \\

$2\times2$
& $30{,}339{,}126$
& $81$
& $374{,}557.11$
& $591.03$
& $608.22$
& $0.972$
& $373{,}325$
& $-0.33\%$
& $375{,}753$
& $+0.32\%$
& $3.80\mathbin{\cdot}10^{-10}$
& $2.266\mathbin{\cdot}10^{-4}$
& $3.713\mathbin{\cdot}10^{-5}$ \\

$3\times2$
& $13{,}156{,}873$
& $729$
& $18{,}047.84$
& $128.56$
& $134.25$
& $0.958$
& $17{,}654$
& $-2.18\%$
& $18{,}503$
& $+2.52\%$
& $9.55\mathbin{\cdot}10^{-11}$
& $2.909\mathbin{\cdot}10^{-4}$
& $2.372\mathbin{\cdot}10^{-5}$ \\

$4\times2$
& $3{,}346{,}495$
& $6{,}561$
& $510.06$
& $22.81$
& $22.58$
& $1.010$
& $428$
& $-16.09\%$
& $592$
& $+16.07\%$
& $4.65\mathbin{\cdot}10^{-11}$
& $6.210\mathbin{\cdot}10^{-4}$
& $1.171\mathbin{\cdot}10^{-4}$ \\

$5\times2$
& $521{,}106$
& $59{,}049$
& $8.82$
& $2.98$
& $2.97$
& $1.003$
& $0$
& $-100.00\%$
& $24$
& $+171.96\%$
& $3.27\mathbin{\cdot}10^{-11}$
& $1.021\mathbin{\cdot}10^{-3}$
& $3.984\mathbin{\cdot}10^{-4}$ \\

$6\times2$
& $49{,}809$
& $531{,}441$
& $0.0937$
& $0.306$
& $0.306$
& $0.999$
& $0$
& $-100.00\%$
& $3$
& $+3{,}100.87\%$
& $3.77\mathbin{\cdot}10^{-11}$
& $7.649\mathbin{\cdot}10^{-3}$
& $2.023\mathbin{\cdot}10^{-4}$ \\

\addlinespace
\multicolumn{14}{l}{\textbf{F1k-4}} \\
\addlinespace[2pt]
$1\times3$
& $38{,}418{,}717$
& $27$
& $1{,}422{,}915.44$
& $10{,}005.40$
& $1{,}170.56$
& $8.548$
& $1{,}397{,}782$
& $-1.77\%$
& $1{,}440{,}616$
& $+1.24\%$
& $6.78\mathbin{\cdot}10^{-8}$
& $2.593\mathbin{\cdot}10^{-3}$
& $6.542\mathbin{\cdot}10^{-4}$ \\

$2\times3$
& $29{,}630{,}608$
& $729$
& $40{,}645.55$
& $604.06$
& $201.47$
& $2.998$
& $38{,}330$
& $-5.70\%$
& $41{,}876$
& $+3.03\%$
& $4.16\mathbin{\cdot}10^{-10}$
& $4.856\mathbin{\cdot}10^{-3}$
& $1.745\mathbin{\cdot}10^{-3}$ \\

$3\times3$
& $11{,}445{,}440$
& $19{,}683$
& $581.49$
& $27.48$
& $24.11$
& $1.140$
& $465$
& $-20.03\%$
& $683$
& $+17.46\%$
& $5.76\mathbin{\cdot}10^{-12}$
& $5.171\mathbin{\cdot}10^{-3}$
& $2.206\mathbin{\cdot}10^{-3}$ \\

$4\times3$
& $2{,}466{,}815$
& $531{,}441$
& $4.64$
& $2.16$
& $2.15$
& $1.002$
& $0$
& $-100.00\%$
& $18$
& $+287.78\%$
& $7.65\mathbin{\cdot}10^{-13}$
& $5.537\mathbin{\cdot}10^{-3}$
& $2.488\mathbin{\cdot}10^{-3}$ \\

\addlinespace
\multicolumn{14}{l}{\textbf{F10k-4}} \\
\addlinespace[2pt]
$1\times3$
& $39{,}520{,}641$
& $27$
& $1{,}463{,}727.44$
& $2{,}241.27$
& $1{,}187.23$
& $1.888$
& $1{,}458{,}336$
& $-0.37\%$
& $1{,}466{,}584$
& $+0.20\%$
& $3.22\mathbin{\cdot}10^{-9}$
& $4.455\mathbin{\cdot}10^{-4}$
& $1.364\mathbin{\cdot}10^{-4}$ \\

$2\times3$
& $33{,}815{,}155$
& $729$
& $46{,}385.67$
& $252.68$
& $215.23$
& $1.174$
& $45{,}659$
& $-1.57\%$
& $47{,}256$
& $+1.88\%$
& $5.58\mathbin{\cdot}10^{-11}$
& $7.711\mathbin{\cdot}10^{-4}$
& $3.580\mathbin{\cdot}10^{-4}$ \\

$3\times3$
& $15{,}178{,}560$
& $19{,}683$
& $771.15$
& $27.96$
& $27.77$
& $1.007$
& $652$
& $-15.45\%$
& $892$
& $+15.67\%$
& $3.39\mathbin{\cdot}10^{-12}$
& $6.591\mathbin{\cdot}10^{-4}$
& $2.783\mathbin{\cdot}10^{-4}$ \\

$4\times3$
& $4{,}040{,}990$
& $531{,}441$
& $7.60$
& $2.76$
& $2.76$
& $1.001$
& $0$
& $-100.00\%$
& $25$
& $+228.78\%$
& $4.67\mathbin{\cdot}10^{-13}$
& $9.722\mathbin{\cdot}10^{-4}$
& $4.555\mathbin{\cdot}10^{-4}$ \\

\addlinespace
\multicolumn{14}{l}{\textbf{F100k-4}} \\
\addlinespace[2pt]
$1\times3$
& $34{,}063{,}120$
& $27$
& $1{,}261{,}597.04$
& $1{,}021.68$
& $1{,}102.21$
& $0.927$
& $1{,}259{,}270$
& $-0.18\%$
& $1{,}263{,}623$
& $+0.16\%$
& $9.00\mathbin{\cdot}10^{-10}$
& $1.389\mathbin{\cdot}10^{-4}$
& $6.193\mathbin{\cdot}10^{-6}$ \\

$2\times3$
& $31{,}318{,}687$
& $729$
& $42{,}961.16$
& $205.23$
& $207.13$
& $0.991$
& $42{,}291$
& $-1.56\%$
& $43{,}574$
& $+1.43\%$
& $4.29\mathbin{\cdot}10^{-11}$
& $1.700\mathbin{\cdot}10^{-4}$
& $6.354\mathbin{\cdot}10^{-5}$ \\

$3\times3$
& $15{,}513{,}787$
& $19{,}683$
& $788.18$
& $28.12$
& $28.07$
& $1.001$
& $685$
& $-13.09\%$
& $904$
& $+14.69\%$
& $3.28\mathbin{\cdot}10^{-12}$
& $2.346\mathbin{\cdot}10^{-4}$
& $1.052\mathbin{\cdot}10^{-4}$ \\

$4\times3$
& $4{,}722{,}082$
& $531{,}441$
& $8.89$
& $2.98$
& $2.98$
& $1.000$
& $0$
& $-100.00\%$
& $26$
& $+192.61\%$
& $3.98\mathbin{\cdot}10^{-13}$
& $5.496\mathbin{\cdot}10^{-4}$
& $1.471\mathbin{\cdot}10^{-4}$ \\

\addlinespace
\multicolumn{14}{l}{\textbf{F1k-5}} \\
\addlinespace[2pt]
$1\times4$
& $37{,}768{,}099$
& $81$
& $466{,}272.83$
& $3{,}583.54$
& $678.61$
& $5.281$
& $457{,}128$
& $-1.96\%$
& $472{,}373$
& $+1.31\%$
& $9.00\mathbin{\cdot}10^{-9}$
& $2.651\mathbin{\cdot}10^{-3}$
& $2.338\mathbin{\cdot}10^{-4}$ \\

$2\times4$
& $31{,}593{,}318$
& $6{,}561$
& $4{,}815.32$
& $103.85$
& $69.39$
& $1.497$
& $4{,}453$
& $-7.52\%$
& $5{,}137$
& $+6.68\%$
& $1.08\mathbin{\cdot}10^{-11}$
& $5.417\mathbin{\cdot}10^{-3}$
& $7.541\mathbin{\cdot}10^{-4}$ \\

$3\times4$
& $13{,}736{,}100$
& $531{,}441$
& $25.85$
& $5.12$
& $5.08$
& $1.008$
& $5$
& $-80.66\%$
& $54$
& $+108.92\%$
& $1.39\mathbin{\cdot}10^{-13}$
& $6.124\mathbin{\cdot}10^{-3}$
& $1.520\mathbin{\cdot}10^{-3}$ \\

\addlinespace
\multicolumn{14}{l}{\textbf{F10k-5}} \\
\addlinespace[2pt]
$1\times4$
& $36{,}011{,}308$
& $81$
& $444{,}584.05$
& $931.39$
& $662.64$
& $1.406$
& $442{,}483$
& $-0.47\%$
& $446{,}603$
& $+0.45\%$
& $6.69\mathbin{\cdot}10^{-10}$
& $5.763\mathbin{\cdot}10^{-4}$
& $4.296\mathbin{\cdot}10^{-5}$ \\

$2\times4$
& $33{,}141{,}355$
& $6{,}561$
& $5{,}051.27$
& $73.93$
& $71.07$
& $1.040$
& $4{,}755$
& $-5.87\%$
& $5{,}311$
& $+5.14\%$
& $4.98\mathbin{\cdot}10^{-12}$
& $9.651\mathbin{\cdot}10^{-4}$
& $1.426\mathbin{\cdot}10^{-4}$ \\

$3\times4$
& $16{,}456{,}380$
& $531{,}441$
& $30.97$
& $5.56$
& $5.56$
& $0.999$
& $9$
& $-70.94\%$
& $60$
& $+93.76\%$
& $1.14\mathbin{\cdot}10^{-13}$
& $9.371\mathbin{\cdot}10^{-4}$
& $3.759\mathbin{\cdot}10^{-4}$ \\

\addlinespace
\multicolumn{14}{l}{\textbf{F100k-5}} \\
\addlinespace[2pt]
$1\times4$
& $34{,}242{,}370$
& $81$
& $422{,}745.31$
& $614.90$
& $646.16$
& $0.952$
& $421{,}417$
& $-0.31\%$
& $424{,}391$
& $+0.39\%$
& $3.22\mathbin{\cdot}10^{-10}$
& $1.376\mathbin{\cdot}10^{-4}$
& $8.489\mathbin{\cdot}10^{-6}$ \\

$2\times4$
& $33{,}668{,}895$
& $6{,}561$
& $5{,}131.67$
& $71.58$
& $71.63$
& $0.999$
& $4{,}868$
& $-5.14\%$
& $5{,}448$
& $+6.16\%$
& $4.52\mathbin{\cdot}10^{-12}$
& $1.950\mathbin{\cdot}10^{-4}$
& $1.950\mathbin{\cdot}10^{-5}$ \\

$3\times4$
& $18{,}281{,}300$
& $531{,}441$
& $34.40$
& $5.87$
& $5.87$
& $1.000$
& $11$
& $-68.02\%$
& $65$
& $+88.96\%$
& $1.03\mathbin{\cdot}10^{-13}$
& $3.063\mathbin{\cdot}10^{-4}$
& $5.446\mathbin{\cdot}10^{-5}$ \\

\addlinespace
\multicolumn{14}{l}{\textbf{F1k-6}} \\
\addlinespace[2pt]
$1\times5$
& $43{,}099{,}798$
& $243$
& $177{,}365.42$
& $1{,}549.96$
& $420.28$
& $3.688$
& $172{,}642$
& $-2.66\%$
& $181{,}124$
& $+2.12\%$
& $1.29\mathbin{\cdot}10^{-9}$
& $2.838\mathbin{\cdot}10^{-3}$
& $1.096\mathbin{\cdot}10^{-4}$ \\

$2\times5$
& $38{,}339{,}437$
& $59{,}049$
& $649.28$
& $27.98$
& $25.48$
& $1.098$
& $536$
& $-17.45\%$
& $773$
& $+19.05\%$
& $5.33\mathbin{\cdot}10^{-13}$
& $5.655\mathbin{\cdot}10^{-3}$
& $2.789\mathbin{\cdot}10^{-4}$ \\

\addlinespace
\multicolumn{14}{l}{\textbf{F10k-6}} \\
\addlinespace[2pt]
$1\times5$
& $42{,}014{,}728$
& $243$
& $172{,}900.12$
& $453.95$
& $414.96$
& $1.094$
& $171{,}727$
& $-0.68\%$
& $173{,}960$
& $+0.61\%$
& $1.17\mathbin{\cdot}10^{-10}$
& $4.687\mathbin{\cdot}10^{-4}$
& $2.792\mathbin{\cdot}10^{-5}$ \\

$2\times5$
& $40{,}899{,}399$
& $59{,}049$
& $692.63$
& $26.43$
& $26.32$
& $1.004$
& $583$
& $-15.83\%$
& $803$
& $+15.93\%$
& $4.17\mathbin{\cdot}10^{-13}$
& $9.322\mathbin{\cdot}10^{-4}$
& $3.463\mathbin{\cdot}10^{-5}$ \\

\addlinespace
\multicolumn{14}{l}{\textbf{F100k-6}} \\
\addlinespace[2pt]
$1\times5$
& $41{,}218{,}322$
& $243$
& $169{,}622.72$
& $417.16$
& $411.00$
& $1.015$
& $168{,}629$
& $-0.59\%$
& $170{,}741$
& $+0.66\%$
& $1.02\mathbin{\cdot}10^{-10}$
& $1.894\mathbin{\cdot}10^{-4}$
& $6.363\mathbin{\cdot}10^{-6}$ \\

$2\times5$
& $42{,}714{,}501$
& $59{,}049$
& $723.37$
& $26.83$
& $26.90$
& $0.997$
& $611$
& $-15.53\%$
& $844$
& $+16.68\%$
& $3.94\mathbin{\cdot}10^{-13}$
& $2.365\mathbin{\cdot}10^{-4}$
& $1.529\mathbin{\cdot}10^{-5}$ \\

\addlinespace
\multicolumn{14}{l}{\textbf{F10k-7}} \\
\addlinespace[2pt]
$1\times6$
& $51{,}722{,}830$
& $729$
& $70{,}950.38$
& $291.67$
& $266.18$
& $1.096$
& $70{,}104$
& $-1.19\%$
& $71{,}706$
& $+1.06\%$
& $3.18\mathbin{\cdot}10^{-11}$
& $6.133\mathbin{\cdot}10^{-4}$
& $6.832\mathbin{\cdot}10^{-6}$ \\

$2\times6$
& $52{,}656{,}296$
& $531{,}441$
& $99.08$
& $9.97$
& $9.95$
& $1.002$
& $58$
& $-41.46\%$
& $150$
& $+51.39\%$
& $3.59\mathbin{\cdot}10^{-14}$
& $1.075\mathbin{\cdot}10^{-3}$
& $2.283\mathbin{\cdot}10^{-5}$ \\

\addlinespace
\multicolumn{14}{l}{\textbf{F1k-8}} \\
\addlinespace[2pt]
$1\times7$
& $71{,}089{,}070$
& $2{,}187$
& $32{,}505.29$
& $370.55$
& $180.25$
& $2.056$
& $31{,}096$
& $-4.34\%$
& $33{,}562$
& $+3.25\%$
& $2.72\mathbin{\cdot}10^{-11}$
& $3.148\mathbin{\cdot}10^{-3}$
& $1.982\mathbin{\cdot}10^{-5}$ \\

\addlinespace
\multicolumn{14}{l}{\textbf{F1k-9}} \\
\addlinespace[2pt]
$1\times8$
& $71{,}563{,}599$
& $6{,}561$
& $10{,}907.42$
& $156.61$
& $104.43$
& $1.500$
& $10{,}342$
& $-5.18\%$
& $11{,}463$
& $+5.09\%$
& $4.79\mathbin{\cdot}10^{-12}$
& $3.243\mathbin{\cdot}10^{-3}$
& $6.811\mathbin{\cdot}10^{-6}$ \\

\addlinespace
\multicolumn{14}{l}{\textbf{F1k-10}} \\
\addlinespace[2pt]
$1\times9$
& $71{,}962{,}170$
& $19{,}683$
& $3{,}656.06$
& $73.22$
& $60.46$
& $1.211$
& $3{,}370$
& $-7.82\%$
& $4{,}006$
& $+9.57\%$
& $1.04\mathbin{\cdot}10^{-12}$
& $3.384\mathbin{\cdot}10^{-3}$
& $3.391\mathbin{\cdot}10^{-6}$ \\

\addlinespace
\multicolumn{14}{l}{\textbf{F1k-11}} \\
\addlinespace[2pt]
$1\times10$
& $72{,}331{,}479$
& $59{,}049$
& $1{,}224.94$
& $37.89$
& $35.00$
& $1.083$
& $1{,}065$
& $-13.06\%$
& $1{,}384$
& $+12.99\%$
& $2.74\mathbin{\cdot}10^{-13}$
& $3.460\mathbin{\cdot}10^{-3}$
& $1.078\mathbin{\cdot}10^{-6}$ \\

\addlinespace
\multicolumn{14}{l}{\textbf{F1k-12}} \\
\addlinespace[2pt]
$1\times11$
& $72{,}648{,}069$
& $177{,}147$
& $410.10$
& $20.87$
& $20.25$
& $1.031$
& $322$
& $-21.48\%$
& $504$
& $+22.90\%$
& $8.25\mathbin{\cdot}10^{-14}$
& $3.668\mathbin{\cdot}10^{-3}$
& $3.455\mathbin{\cdot}10^{-7}$ \\

\end{longtable}
\end{landscape}

	\bibliography{ref}
	\bibliographystyle{alpha}

\end{document}